\documentclass[11pt]{article}
\usepackage[top=99pt, bottom=75pt, left=72pt, right=70pt]{geometry}
\usepackage[nottoc]{tocbibind} 
\usepackage[utf8]{inputenc}
\usepackage{amsthm}
\usepackage{amsmath}
\usepackage{amssymb}
\usepackage{amscd,amsxtra}
\usepackage{MnSymbol}
\usepackage{mathtools}
\usepackage{tensor}
\usepackage{faktor}
\usepackage{dsfont}
\usepackage[all]{xy}
\usepackage{graphicx}
\usepackage[rgb]{xcolor}
\usepackage{tikz, tikz-cd}
\usetikzlibrary{arrows, matrix, intersections, math, calc, decorations.pathmorphing, decorations.markings, positioning}
\usepackage{enumitem}
\usepackage{hyperref}
\definecolor{R}{RGB}{255, 0, 34}
\definecolor{B}{RGB}{0, 85, 238}
\hypersetup{colorlinks = true,
            linkcolor = R,
            urlcolor  = B,
            citecolor = B,
            anchorcolor = B}

\newtheorem{theorem}{Theorem}
\newtheorem{lemma}{Lemma}

\newtheorem{proposition}{Proposition}  
\newtheorem{corollary}{Corollary} 
\newtheorem{conjecture}{Conjecture}
\newtheorem{definition}{Definition}
\newtheorem{remark}{Remark}

\numberwithin{theorem}{subsection}
\numberwithin{proposition}{subsection}
\numberwithin{lemma}{subsection}
\numberwithin{claim}{subsection}
\numberwithin{corollary}{subsection}
\numberwithin{conjecture}{subsection}
\numberwithin{definition}{subsection}
\numberwithin{remark}{subsection}
\newtheorem{example}[theorem]{Example}

\newcommand{\beq}{\begin{equation}}
\newcommand{\eeq}{\end{equation}}
\newcommand{\beqa}{\begin{eqnarray}}
\newcommand{\eeqa}{\end{eqnarray}}
\newcommand{\beaa}{\begin{eqnarray*}}
\newcommand{\ben}{\begin{eqnarray*}}
\newcommand{\eaa}{\end{eqnarray*}}
\newcommand{\een}{\end{eqnarray*}}

\newcommand{\CC}{\mathbb{C}}
\newcommand{\ZZ}{\mathbb{Z}}

\newcommand{\RR}{\mathbb{R}}

\newcommand{\ii}{\mathbf{i}}
\newcommand{\one}{\mathbf{1}}
\newcommand{\op}{\operatorname}

\newcommand{\T}{\mathcal{T}}

\renewcommand{\O}{\mathcal{O}}

\newcommand{\M}{\mathcal{M}}

\title{Dubrovin conjecture and the second structure connection} %% Article title

\author{John Alexander Cruz Morales \and Todor Milanov}

\begin{document}

\maketitle

\begin{abstract}
We give a reformulation of the Dubrovin conjecture about the
semisimplicity of quantum cohomology in 
terms of the so-called second structure connection of quantum
cohomology. The key ingredient in our work is the notion of a
twisted reflection vector which allows us to give an elegant description
of the monodromy data of the quantum connection in terms of the
monodromy data of its Laplace transform.
\end{abstract}

%%Graphical abstract
%\begin{graphicalabstract}
%\includegraphics{grabs}
%\end{graphicalabstract}

%%Research highlights
%\begin{highlights}
%\item Research highlight 1
%\item Research highlight 2
%\end{highlights}

%% Keywords

\setcounter{tocdepth}{2}
\tableofcontents

%\end{frontmatter}

%% Add \usepackage{lineno} before \begin{document} and uncomment 
%% following line to enable line numbers
%% \linenumbers

%% main text
%%

\section{Introduction}

\subsection{Semisimple quantum cohomology}
Let $X$ be a smooth projective variety of complex dimension $D$. The Gromov--Witten (GW) invariants of $X$ are defined via the intersection theory on the moduli space of stable maps $\overline{\M}_{g,n}(X,d)$ (see Section \ref{sec:q_coh} for more details). The structure of the GW invariants is best understood in genus $g=0$. Namely, the entire information is contained in a certain deformation of the classical cup product, known as the {\em quantum cup product}.  Let us fix a homogeneous basis $\phi_i$ ($1\leq i\leq N$) of $H^*(X,\CC)$.  Then the quantum cup product $\bullet$ is defined by
\ben
(\phi_i\bullet \phi_j,\phi_k) =
\sum_{l=0}^\infty \sum_d
\frac{Q^d}{l!}
\langle\phi_i,\phi_j,\phi_k,t,\dots,t\rangle_{0,3+l,d},
\een
where the second sum is over all effective curve classes $d\in H_2(X,\ZZ)$, $Q=(q_1,\dots,q_r)$ are the so-called {\em Novikov variables}, and $t=\sum_{i=1}^N t_i\phi_i\in H^*(X,\CC)$. The Novikov variables correspond to a choice of an ample basis $p_1,\dots,p_r$ of $H^2(X,Z)\cap H^{1,1}(X,\CC)$ and
$Q^d:= q_1^{\langle p_1,d\rangle}  \cdots q_r^{\langle p_r,d\rangle}$. Let us assume that $\phi_1=1$ and $\phi_{i+1}=p_i$. Then the structure constants of the quantum cup product belong to the following ring of formal power series
\ben
\CC[\![Q,t]\!]:=\CC[\![q_1 e^{t_2},\dots, q_r e^{t_{r+1}},
t_{r+2},\dots,t_N]\!],
\een
where the fact that there is a dependence on $q_ie^{t_{i+1}}$ is a consequence of the so-called {\em divisor equation} (see \cite{Ma}). The quantum cup product turns $QH(X,\CC):=H^*(X,\CC[\![Q,t]\!])$ into a commutative associative algebra known as the {\em big quantum cohomology} of $X$. We say that the quantum cohomology is {\em semisimple} if  $QH(X,\CC)$ is a semisimple algebra, or equivalently the operators $\phi\bullet$ of quantum multiplication by $\phi\in H^*(X,\CC)$ are not nilpotent. In the limit $Q\to 0$, the quantum product becomes the classical cup product and since $\phi\cup$ is always nilpotent (for $\operatorname{deg}(\phi)>0$), we see that semi-simplicity is an indication that the manifold $X$ has sufficiently many rational curves. From that point of view it is very interesting to classify manifolds with semisimple quantum cohomology. It is an observation of Alexey Bondal that semi-simplicity can be characterized using the language of derived categories. In his ICM talk in 1998, following Bondal's ideas, Dubrovin was able to formulate a precise conjecture which is now known as the Dubrovin conjecture (see \cite{Dubrovin1998} and Conjecture \ref{conj:Du1998}). The goal of this paper is to answer a question which was raised by the second author in his joint work \cite{MilanovXia}. Namely, there is a conjectural description (see \cite{MilanovXia}, Conjecture 1.6) of the so-called reflection vectors in quantum cohomology in terms of exceptional collections in the derived category. More precisely, Milanov--Xia were able to construct reflection vectors in the quantum cohomology of the blowup of a manifold relying only on certain vanishing results for GW invariants. The resulting formulas were very similar to the formulas for the central connection matrix conjectured by Galkin--Golyshev--Iritani in \cite{GGI2016} and later on by Cotti--Dubrovin--Guzzetti \cite{CDG2024}. The question is whether one can reformulate Dubrovin's conjecture in terms of reflection vectors. As expected, the answer is yes and in this paper we would like to work out the precise relation between the reflection vectors and the monodromy data which enters the Dubrovin conjecture. Our main message is that by constructing a basis of reflection vectors in quantum cohomology one can obtain a proof of the $\Gamma$-Conjecture II of Galkin--Golyshev--Iritani (see \cite{GGI2016}) or equivalently the refined Dubrovin conjecture of Cotti--Dubrovin--Guzzetti (see \cite{CDG2024}, Conjecture 5.2). In the rest of this introduction we would like to formulate our results. 

\subsection{Quantum differential equations}\label{sec:qde}

From now on we are going to assume that quantum cohomology is semisimple and convergent. The latter means that there exists a domain $M\subseteq  H^*(X,\CC)$, such that the formal power series representing the structure constants are convergent. Let us introduce the following two linear operators:
\ben
\theta:H^*(X,\CC)\to H^*(X,\CC),\quad
\theta(\phi_i) = \Big(
\frac{D}{2}-\op{deg}_\CC(\phi_i) \Big) \phi_i
\een
and
\ben
\rho:H^*(X,\CC)\to H^*(X,\CC),\quad
\rho(\phi) = c_1(TX)\cup \phi, 
\een
where $\op{deg}_\CC(\phi):=k$ for $\phi\in H^{2k}(X,\CC)$ and $\cup$ is the classical cup product. The {\em quantum differential equations} are by definition the differential equations of the following system of ODEs:
\beqa\label{t-qde}
z\partial_{t_i} J(t,z) & = & A_i(t) J(t,z),\quad 1\leq i\leq N,\\
\label{z-qde}
(z\partial_z + E) J(t,z) & = & \theta J(t,z),
\eeqa
where $J(t,z)\in H^*(X,\CC)$, $A_i(t)=\phi_i\bullet$ is the operator of quantum multiplication by $\phi_i$, and
\ben
E = c_1(X) + \sum_{i=1}^N \Big(
1 - \mathrm{deg}_\CC(\phi_i) \Big)
t_i \frac{\partial}{\partial t_i}
\een
is the {\em Euler vector field}. The differential equations \eqref{t-qde}--\eqref{z-qde} can be viewed also as the equations defining the horizontal sections of a connection on the vector bundle $TM\times \CC^*\to M\times \CC^*$ which is sometimes called {\em quantum connection} or {\em Dubrovin connection}.

The differential equation with respect to $z$ has two singularities: regular at $z=\infty$ and irregular at $z=0$. Near $z=\infty$ there is a geometric way to construct a fundamental solution. Namely, let us define $S(t,z)=1+S_1(t)z^{-1}+S_2(t)z^{-2}+\cdots$ where $S_k(t)\in \op{End}(H^*(X,\CC))$ are linear operators defined by
\ben
(S_k(t)\phi_i,\phi_j) = \sum_{l=0}^\infty\sum_d
\frac{Q^d}{l!} \langle \phi_i \psi^{k-1},\phi_j,t,\dots,t\rangle_{0,2+l,d}.
\een
Then $S(t,z)z^{\theta}z^{-\rho}$ is a solution to the system \eqref{t-qde}--\eqref{z-qde}.

The singularity at $z=0$ has an interesting Stokes phenomenon. Namely, suppose that $(u_1,\dots,u_N)$ are the canonical coordinates defined in a neighborhood of some semisimple point $t^\circ\in M$. By definition, the quantum product and the Poincaré pairing become diagonal:
\ben
\frac{\partial}{\partial u_i} \bullet
\frac{\partial}{\partial u_j} = \delta_{ij} \frac{\partial}{\partial u_j},\quad
\Big(\frac{\partial}{\partial u_i} ,
\frac{\partial}{\partial u_j}\Big) = \delta_{ij}/\Delta_i, 
\een
where $\Delta_i\in \O_{M,t^\circ}$ are some holomorphic functions.  
Let us define the linear map 
\beq\label{Psi}
\Psi: \CC^N\to H^*(X,\CC),\quad 
\Psi(e_i) =\sqrt{\Delta_i} \frac{\partial}{\partial u_i} = \sum_{a=1}^N
\sqrt{\Delta_i} \frac{\partial t_a}{\partial u_i} \, \phi_a.
\eeq
We may think of $\Psi$ as a $N\times N$ matrix with entries $\Psi_{ai}=\sqrt{\Delta_i} \frac{\partial t_a}{\partial u_i}$. Let $U=\op{diag}(u_1,\dots,u_N)$ be the diagonal matrix. There exists a unique formal asymptotic solution to \eqref{t-qde}--\eqref{z-qde} of the form $\Psi R(t,z) e^{U/z}$ where $R(t,z)=1+R_1(t) z+R_2(t)z^2+\cdots$, $R_k(t)$ are $N\times N$ matrices. The matrices $R_k(t)$ are determined uniquely by plugging in the ansatz $\Psi R(t,z)e^{U/z}$ into \eqref{t-qde}--\eqref{z-qde} and comparing the coefficients in front of the powers of $z$. This gives us a recursion relation for the coefficients $R_k(t)$ which turns out to have a unique solution. Moreover, the solution automatically satisfies $R(t,-z)^t R(t,z)=1$ where ${}^t$ denotes the standard transposition of matrices. Suppose that the semisimple point $t^\circ$ is generic and that the coordinate neighbourhood is so small that $u_i\neq u_j$ for $i\neq j$. The rays in the $z$-plane of the form $\ii (u_i-u_j) \RR_{>0}$ where $\ii:=\sqrt{-1}$ and $i\neq j$, are called the {\em Stokes rays}. Let $0\in\ell\subset \CC$ be a line not parallel to any of the Stokes rays. Let us fix an orientation of $\ell$ by choosing a unit vector $e^{\ii\phi}\in \ell$, $\op{Arg}(\phi)\in [0,2\pi)$. Following Dubrovin (see \cite{Du1999}) we will say that $\ell$ is an {\em admissible oriented line}. The line $\ell$ splits the $z$-plane into right $\Pi_{\rm right}$ and left $\Pi_{\rm left}$  half-planes: if we stand at $z=0$ and look in the direction of $e^{\ii \phi}$ then $\Pi_{\rm left}$ is on our left and $\Pi_{\rm right}$ is on our right. There exists a unique pair of solutions $X_{\rm right}(t,z)$ and $X_{\rm left}(t,z)$ to \eqref{t-qde}--\eqref{z-qde} holomorphic respectively for $z\in \Pi_{\rm right}$ and $z\in \Pi_{\rm left}$   which are asymptotic to $\Psi R(t,z) e^{U/z}$ as $z\to 0$. These solutions extend analytically in $z$ along the positive part $\ell_+$ of $\ell$. In particular, we obtain 3 holomorphic solutions to \eqref{t-qde}--\eqref{z-qde} along the positive half $\ell_+$ of $\ell$ which must be related as follows:
\ben
X_{\rm left}(t,z)= X_{\rm right}(t,z) V_+,\quad
X_{\rm left}(t,z) = S(t,z) z^\theta z^{-\rho} C^{-1},\quad
\forall z\in \ell_+,
\een
where $V_+$ and $C^{-1}$ are some constant matrices called respectively, the {\em Stokes matrix} and the {\em central connection matrix}. Here we follow the conventions of Dubrovin in \cite{Du1999} but there is a small discrepancy coming from the fact that the spectral parameter $z$ in our paper corresponds to $z^{-1}$ in \cite{Du1999}. The involution $z\mapsto z^{-1}$ switches the roles of left and right so $V_+^{-1}$ corresponds to the Stokes matrix $S$ in \cite{Du1999} while the central connection matrix here is defined in such a way that it coincides with the central connection matrix in \cite{Du1999}. We refer also to Remark \ref{rem:Du1999} for more precise comparison. The refined Dubrovin conjecture (see \cite{CDG2024}, Conjecture 5.2) consists of 3 parts. The first part says that the big quantum cohomology of $X$ is semisimple if and only if the bounded derived category $D^b(X)$ has a full exceptional collection. The second part says that every admissible line $\ell$ determines a full exceptional collection $(E_1,\dots,E_N)$ which determines uniquely the Stokes matrix $V_+$ and the central connection matrix $C^{-1}$. Finally, the 3rd part of the conjecture gives very precise formulas for both $V_+$ and $C$, that is, 
the $(i,j)$-entry of $V_+$ is
\ben
V_{+,ij}=\chi(E_i,E_j),\quad 1\leq i,j\leq N
\een
and the $i$-th column of $C^{-1}$ is
\ben
C^{-1}(e_i)=
\frac{\ii^{\overline{D}}}{(2\pi)^{D/2}}
\widehat{\Gamma}^-_X \cup e^{-\pi\ii\rho}\cup 
\operatorname{Ch}(E_i),
\een
where $\overline{D}\in \{0,1\}$ is the remainder of the division of $D$ by $2$, $\widehat{\Gamma}^-_X=\prod_\delta \Gamma(1-\delta)$ is the so-called {\em gamma class} of $X$, and $\operatorname{Ch}(E)=\sum_\epsilon e^{2\pi\ii \epsilon}$ is the Chern character of $E$. Here the products are over the Chern roots $\delta$ and $\epsilon$ of respectively the holomorphic tangent bundle $TX$ and the complex vector bundle $E$. Just like in the case of $\Psi$, we think of $C^{-1}$ as a linear map $C^{-1}:\CC^N\to H^*(X,\CC)$. We refer to Section \ref{sec:rdc} for more details.

The refined version of the conjecture still requires that the manifold $X$ is Fano. However, as it was pointed out by Arend Bayer in \cite{Ba2004} (see also the recent work by Hiroshi Iritani \cite{Ir2023}), by using the blowup operation we can construct many examples of non-Fano manifolds for which the first part of the Dubrovin conjecture holds. Moreover, the recent work by Milanov--Xia (see \cite{MilanovXia}) gives an indication that the blowup operation preserves the remaining two parts of the Dubrovin conjecture. Therefore, it is quite plausible that the Fano condition is redundant.

\subsection{Reflection vectors}\label{sec:rv_qcoh}

Suppose that the big quantum cohomology is semisimple and convergent. The solutions to the quantum differential equations can be represented by complex oscillatory integrals of the following form:
\ben
J(t,z) :=
\frac{1}{\sqrt{2\pi}}\,
(-z)^{m-1/2} 
\int_{\Gamma} 
e^{\lambda/z} I^{(m)}(t,\lambda)d\lambda,
\een
where $m\in \CC$ is a complex number and the semi-infinite integration cycle is chosen in such a way that the integral is convergent. It is easy to check that the above integral solves the quantum differential equations \eqref{t-qde}--\eqref{z-qde} iff the integrand $I^{(m)}(t,\lambda)$ satisfies the following system of ODEs:
\begin{align}
  \label{t-lqde}
\partial_{t_i} I^{(m)}(t,\lambda) & = -  
                         (\lambda-E\bullet)^{-1}
                         (\phi_i \bullet) (\theta-m-1/2)\  I^{(m)}(t,\lambda),
  \\
  \label{la-lqde}
  \partial_\lambda I^{(m)}(t,\lambda)& = 
(\lambda-E\bullet_t)^{-1}(\theta-m-1/2)\,
I^{(m)}(t,\lambda).
\end{align}
This is a system of differential equations for the horizontal sections of a connection $\nabla^{(m)}$ on the trivial bundle
\ben
(M\times \CC)'\times H^*(X,\CC) \to (M\times \CC)',
\een
where 
\ben
(M\times \CC)'=\{ (t,\lambda)\ |\ \det (\lambda-E\bullet_t)\neq 0\}.
\een
The hypersurface $\det (\lambda-E\bullet_t)=0$ in $M\times \CC$ is
called the {\em discriminant}. The connection $\nabla^{(m)}$ is known as the second structure connection. In the case when $m=0$ or $-1$, the connection was used by Dubrovin to define the monodromy group of a Frobenius manifold (see \cite{Du1996}). However, it became clear shortly afterwards that it is important to study the entire family, that is, allow $m$ to be any complex number (see \cite{MM1997} and \cite{Du2004}).

The space of solutions to \eqref{t-lqde}--\eqref{la-lqde} is quite interesting. In the examples of mirror symmetry the second structure connection of quantum cohomology can be identified with a Gauss--Manin connection. Therefore, the solutions to  \eqref{t-lqde}--\eqref{la-lqde} should be thought as period integrals. In particular, by using $\nabla^{(m)}$ we can introduce many of the ingredients of Picard--Lefschetz theory. This was done by Dubrovin (see \cite{Du2004}, Section 4). He called the solutions to \eqref{t-lqde}--\eqref{la-lqde} {\em twisted periods} because their properties are very similar to the period integrals in Givental's twisted Picard--Lefschetz theory \cite{Giv1988}. Motivated by the work of Givental in \cite{Giv2003}, the second author introduced in \cite{Milanov:p2} the following fundamental solution to \eqref{t-lqde}--\eqref{la-lqde}:
\beq\label{fundamental_period_qcoh}
I^{(m)}(t,\lambda) = \sum_{k=0}^\infty (-1)^k S_k(t) \widetilde{I}^{(m+k)}(\lambda),
\eeq
where
\beq\label{calibrated_period_qcoh}
\widetilde{I}^{(m)}(\lambda) = e^{-\rho \partial_\lambda \partial_m}
\Big(
\frac{\lambda^{\theta-m-\frac{1}{2}} }{ \Gamma(\theta-m+\frac{1}{2}) }
\Big).
\eeq
Note that both $I^{(m)}(t,\lambda)$ and $\widetilde{I}^{(m)}(\lambda)$
take values in $\op{End}(H^*(X,\CC))$.  The second structure connection has a Fuchsian singularity at infinity, therefore the series $I^{(m)}(t,\lambda)$ is convergent and it defines a multi-valued analytic function in the complement to the discriminant. There are many ways to choose a fundamental solution but what makes the above choice special is the specific choice of building blocks, that is, the calibrated periods \eqref{calibrated_period_qcoh} while the standard approach would be monomials in $\lambda$. The existence of such decomposition follows from Givental's formalism of quantized symplectic transformations and their actions on vertex operators (see \cite{Giv2003}, Section 5, especially Theorem 2). Although Dubrovin already knew that one can import concepts from singularity theory to quantum cohomology, somehow the above choice of a fundamental solution makes the parallel with singularity theory much more visible (at least to the authors). 

 Let us choose a base point $t^\circ\in M$ such that
$\operatorname{Re} u_i(t^\circ)\neq \operatorname{Re} u_j(t^\circ)$
for $i\neq j$. Then $\operatorname{Re}  u_i(t)\neq
\operatorname{Re}  u_j(t)$ for $i\neq j$ for all $t$ sufficiently close to $t^\circ$. Let $\lambda^\circ$ be a positive real number such that $\lambda^\circ>|u_i(t^\circ)|$ for all $i$. We define the $m$-{\em twisted period} vectors $I_a^{(m)}(t,\lambda):=I^{(m)}(t,\lambda) a$ where $a\in H^*(X,\CC)$ and the value depends on the choice of a reference path avoiding the discriminant from $(t^\circ,\lambda^\circ)$ to $(t,\lambda)$. Note that at $(t^\circ,\lambda^\circ)$ the only ambiguity is in the choice of the value for the calibrated periods, that is, we need to specify a branch of $\log \lambda$ when $\lambda$ is close to $\lambda^\circ$. Since $\lambda^\circ$ is a positive real number we simply take the principal branch of the logarithm.

Let us introduce the following pairings $h_m: H^*(X,\CC)\times H^*(X,\CC) \to \CC$ 
\beq\label{herm_pairing-qcoh}
h_m(a,b):= (I^{(m)}_a(t,\lambda),(\lambda-E\bullet) I^{(-m)}_b(t,\lambda)).
\eeq
Using the differential equations of $\nabla^{(\pm m)}$ it is easy to
check that $h_m(a,b)$ is independent of $t$ and $\lambda$. It turns out that there is an explicit formula
for $h_m$ in terms of the Hodge grading operator $\theta$ and the
nilpotent operator $\rho$. Let us recall
the so-called {\em Euler pairing}
\beq\label{euler_p-qcoh}
\langle a, b\rangle :=\frac{1}{2\pi} ( a, e^{\pi \mathbf{i}\theta}
e^{\pi\mathbf{i}\rho} b),\quad a,b\in H^*(X,\CC).
\eeq
As a byproduct of the proof of Theorem \ref{thm:refl_cc} we will get the following simple formula: 
\ben
h_m(a,b) = q \langle a, b\rangle + q^{-1} \langle b, a\rangle, 
\een
where  $q:=e^{\pi\mathbf{i} m}$. The above formula shows that $h_m$ is the analogue of the $\ZZ[q,q^{-1}]$-bilinear intersection form in twisted Picard--Lefschetz theory (see \cite{Giv1988}, Section 3). Furthermore, let us fix a reference path from $(t^\circ,\lambda^\circ)$ to a point $(t,\lambda)$ sufficiently close to a generic point $b$ on the discriminant. The local equation of the discriminant near $b$ has the form $\lambda=u_i(t)$ where $u_i(t)$ is an eigenvalue of $E\bullet$.  It turns out that the set of all $a\in H^*(X,\CC)$ such that $I^{(m)}_a(t,\lambda)$ is analytic at $\lambda=u_i(t)$ is a codimension 1 subspace of $H^*(X,\CC)$. Suppose that $m\nin \tfrac{1}{2}+\ZZ$, then there is a 1-dimensional subspace of vectors $\beta\in H^*(X,\CC)$ such that
$(\lambda-u_i(t))^{m+1/2} I^{(m)}_\beta(t,\lambda) $ is analytic at $\lambda=u_i(t)$ and the value at $\lambda=u_i(t)$ belongs to $\CC \Psi(e_i)$ where $\Psi$ is the map \eqref{Psi}.  Therefore, for the given reference path and an arbitrary choice of $\log (\lambda-u_i(t))$ there is a uniquely defined vector $\beta=\beta(m)$ such that
\beq\label{period_ui}
(\lambda-u_i(t))^{m+1/2} I^{(m)}_\beta(t,\lambda) =
\frac{\sqrt{2\pi}}{
  \Gamma(-m+\tfrac{1}{2})}\, \Psi(e_i) + O(\lambda-u_i(t)),
\eeq
where the coefficient in front of $\Psi(e_i)$ is such that $h_m(\beta(m),\beta(-m))=q+q^{-1}$.
Note that the choice of a reference path and a branch of $\log(\lambda-u_i(t))$ determines $\beta(m)$ for all $m\nin \tfrac{1}{2}+\ZZ$, that is, we have a map
\ben
\beta:\CC\setminus{\{\tfrac{1}{2}+\ZZ\}}\to H^*(X,\CC).
\een
Using that $\partial_\lambda I^{(m)}(t,\lambda)=I^{(m+1)}(t,\lambda)$ we get that this map is periodic:  $\beta(m+1)=\beta(m).$ It will follow from our results that $\beta$ is a trigonometric polynomial, that is, $\beta\in H^*(X,\CC[q^2,q^{-2}])$. 
If we change the value of the logarithmic branch $\log(\lambda-u_i(t))\mapsto \log(\lambda-u_i(t))+2\pi\ii$, then $\beta(m)\mapsto -q^{-2}\beta(m)$. Therefore, for a fixed reference path the value of $\beta(m)$ is fixed up to a factor in the spiral   $(-q^{-2})^\ZZ=\{(-1)^kq^{-2k}\ |\ k\in \ZZ\}$. We will say that $\beta$ is a {\em twisted reflection vector} corresponding to the given reference path. We usually suppress the dependence on the logarithmic branch if the choice is irrelevant or it is clear from the context. The following formula for the local monodromy of $\nabla^{(m)}$ justifies our terminology: 
\ben
a\mapsto a-q^{-1} h_m(a,\beta(-m)) \beta(m),\quad a\in H^*(X,\CC),
\een
that is, the local monodromy is a complex reflection whose fixed points locus is the hyperplane orthogonal to $\beta(-m)$. We refer to Section \ref{sec:local_mon} for more details and for more general settings, i.e., we can introduce twisted reflection vectors for any semisimple Frobenius manifold.

\subsection{Monodromy data and reflection vectors}
\label{sec:mdrv}
We continue to work in the settings from the previous two sections. Let $\ell$ be an admissible oriented line (see Section \ref{sec:qde}) with orientation $e^{\ii\phi}$. By definition $\eta=\ii e^{\ii\phi}$  is not parallel to any of the differences $u_i(t^\circ)-u_j(t^\circ)$ for $i\neq j$. We will refer to $\eta$ as an {\em admissible direction}. 
We will consider only $t\in M$ sufficiently close to $t^\circ$, such that $\eta$ is an admissible direction for $t$, that is, $\eta$ is not parallel to $u_i(t)-u_j(t)$ for $i\neq j$. 
Our choice of a reference point $(t^\circ,\lambda^\circ)$ is such that the real line with its standard orientation is an admissible oriented line. The corresponding admissible direction is $\eta^\circ:=\ii$. Any other admissible direction $\eta$ will be equipped with a reference path to $\eta^\circ$ such that 
\beq\label{adm_Arg}
\log \eta =  \ii\, \operatorname{Arg}(\eta),\quad 
-\frac{\pi}{2} <\operatorname{Arg}(\eta)\leq 
\frac{3\pi}{2}.
\eeq
In other words, we require that our reference path is an arc on $\mathbb{S}^1$ between $\eta^\circ$ and $\eta$ such that if we continuously vary $x$ along it, then $\operatorname{Arg}(x)$ varies in the interval $(-\tfrac{\pi}{2},\tfrac{3\pi}{2}]$. 
Let us construct a system of reference paths $C_1(\eta),\dots,C_N(\eta)$ corresponding to $\eta$. Each $C_i(\eta)$ starts at $\lambda=u_i$, approaches the circle $|\lambda|=\lambda^\circ$ in the direction of $\eta$, after hitting the circle at some point $\lambda^i(\eta)$ the path continues clockwise along the circle arc from $\lambda^i(\eta)$ to $\lambda^\circ(\eta):= -\ii \eta \lambda^\circ$, and finally, by continuously deforming the direction $\eta$ to $\eta^\circ$ along the reference path for $\eta$, the path connects $\lambda^\circ(\eta)$ and $\lambda^\circ(\eta^\circ)=\lambda^\circ$ -- see Figure \ref{fig:rp}. Note that our requirement for the reference path of $\eta$ guarantees that the path $C_i(\eta)$ ($1\leq i\leq N)$ does not wind around the circle $|\lambda|=\lambda^\circ$. 
%Indeed, the length of the arc on $C_i(\eta)$ connecting $\lambda^i(\eta)$ and $\lambda^\circ(\eta)$ is $<\pi \lambda^\circ$  and the length of the arc from $\lambda^\circ(\eta)$ to $\lambda^\circ$ is also $<\pi \lambda^\circ$ because the reference arc between $\eta$ and $\eta^\circ=\ii$ has length $<\pi$! We get that the part of $C_i(\eta)$ that belongs to the circle $|\lambda|=\lambda^\circ$ has length $<2\pi \lambda^\circ$ so it does not have enough length to wind around the circle. 
Moreover, if $\lambda\in C_i(\eta)$ is sufficiently close to $u_i$, then $\lambda-u_i = s\eta$ for some positive real number $s$ and we have a natural choice of a logarithmic branch: $\log(\lambda-u_i):=\ln(s) +\log(\eta)$.  Therefore, as it was explained in Section \ref{sec:rv_qcoh}, we may choose a twisted reflection vector $\beta_i(m)$. In other words, each admissible direction determines a set of twisted reflection vectors $(\beta_1(m),\dots,\beta_N(m))$. Furthermore, the admissible direction determines the following order of the eigenvalues $u_1,\dots,u_N$ of $E\bullet$: we say that $u_i<u_j$ if $u_j$ is on the RHS of the line through $u_i$ parallel to $\eta$ where RHS means that we have to stand at $u_i$ and look in the direction $\eta$. For example, for the standard admissible direction $\eta^\circ=\ii$, $u_i<u_j$ would mean that $\op{Re}(u_i)<\op{Re}(u_j)$. We will refer to the order as the {\em lexicographical order} determined by $\eta$. Let us assume that the enumeration of the eigenvalues $u_1,\dots,u_N$ is according to the lexicographical order, that is, $u_i<u_j$ iff $i<j$. Our main result can be stated as follows.
\begin{theorem}\label{thm:t1}
  Let $\eta$ be an admissible direction and assume that the eigenvalues $u_1,\dots,u_N$ of the operator $E\bullet$ are enumerated according to the lexicographical order corresponding to $\eta$. Then the following statements hold.
\begin{enumerate}
\item[a)] The reflection vectors $\beta_i(m)$ ($1\leq i\leq N$) are
independent of $m$ and the Gram matrix of the Euler pairing \eqref{euler_p-qcoh} is upper-triangular
\ben
\langle \beta_i, \beta_j\rangle = 0 \quad \forall i>j,
\een 
with $1$'s on the diagonal: $\langle \beta_i, \beta_i\rangle = 1$.
\item[b)]
The pairing $h_m$ can be computed by the following formula:
\ben
h_m(a,b)= 
q \langle a, b\rangle + 
q^{-1} \langle b, a\rangle,\quad \forall a,b\in H,
\een
where $\langle\ ,\ \rangle$ is the Euler pairing \eqref{euler_p-qcoh}.
\item[c)]
The inverse Stokes matrix $V_+^{-1}$ coincides with the Gram matrix of the Euler pairing \eqref{euler_p-qcoh}  in the basis $\beta_i$  ($1\leq i\leq N$).
\item[d)]
The $(i,j)$-entry of the central connection matrix is related to the components of the reflection vectors by the following formula:
\ben
C_{ij}=\frac{1}{\sqrt{2\pi}} \, (\beta_i,\phi_j).
\een
\end{enumerate}
\end{theorem}
In fact our result is more general. The above theorem can be formulated in the settings of semisimple Frobenius manifolds. Under an additional technical assumption, i.e., we assume that the Frobenius manifold has a calibration for which the grading operator is a Hodge grading operator (see Definition \ref{def:hgo}), we prove that the conclusions of the above theorem remain true (see Theorem \ref{thm:refl_cc}).

Using Theorem \ref{thm:t1} we can answer the question raised in \cite{MilanovXia}. Following the analogy with singularity theory (see \cite{AGuV, Eb2007}), we introduce the concept of a distinguished system of reference paths (see Definition \ref{def:db}).
Let us recall the Iritani's integral structure map (see \cite{Iritani}) 
$\Psi_Q: K_0(X)\rightarrow H^*(X,\mathbb{C})$ defined by 
\ben
\Psi_Q(E):= (2\pi)^{\tfrac{1-D}{2} }
\widehat{\Gamma}(X) \cup e^{-\sum_{i=1}^r p_i \log q_i} \cup \operatorname{Ch}(E),
\een 
where  $Q=(q_1, \dots, q_r)$ are the Novikov variables corresponding to an ample basis $p_1,\dots,p_r$ of $H^2(X,\mathbb{Z})\cap H^{1,1}(X,\CC)$. We have the following relation: 
\beq\label{Psi_eu}
\langle \Psi_Q(E),\Psi_Q(F)\rangle=\chi(E,F),\quad E,F\in K^0(X)
\eeq
which justifies why we refer to \eqref{euler_p-qcoh} as the Euler pairing. 
\begin{theorem}\label{thm:t2}
Let $\beta_1,\dots,\beta_N$ be the reflection vectors corresponding to a distinguished system of reference paths. Parts (2) and (3) of the refined Dubrovin conjecture, i.e., Conjecture \ref{conj:rdc}, hold if and only if there exists a full exceptional collection $(F_1,\dots,F_N)$ such that $\beta_i=\Psi_Q(F_i)$ for all $i$.  
\end{theorem}
Theorem \ref{thm:t2} is proved in Section \ref{sec:rv-dc} (see Theorem \ref{thm:refl_Du}). The definition of a reflection vector is the same as the definition of a twisted reflection vector except that we require $m\in \ZZ$. Since the fundamental group of $\CC\setminus{\{u_1^\circ,\dots,u_N^\circ\}}$ is generated by simple loops corresponding to the reference paths $C_i(\eta)$ where $\eta$ is an admissible direction, we have the following interesting corollary.
\begin{corollary}\label{cor:trv}
The set of all twisted reflection vectors is a subset of
\ben
H^*(X,\CC[q^2,q^{-2}])=
\CC[q^2,q^{-2}]\beta_1+\cdots + \CC[q^2,q^{-2}]\beta_N,
\een
where $\beta_1,\dots,\beta_N$ are the reflection vectors corresponding to a distinguished system of reference paths. 
In addition, if the manifold $X$ satisfies the refined Dubrovin conjecture, then 
the set of twisted reflection vectors is a subset in the $\ZZ[q^2,q^{-2}]$-lattice
$\ZZ[q^2,q^{-2}]\beta_1+\cdots + \ZZ[q^2,q^{-2}]\beta_N$. 
\end{corollary}
For the proof, thanks to the braid group action on the set of distinguished systems of reference paths (see Section \ref{sec:rv-dc}), we may assume that $\beta_1,\dots,\beta_N$ are the reflection vectors corresponding to the reference paths $C_1(\eta),\dots,C_N(\eta)$ for some admissible direction $\eta$.  Note that every twisted reflection vector is obtained from some $\beta_i$ by a sequence of local monodromy transformations $M_j^{\pm 1}$ ($1\leq j\leq N$) where $M_j$ is the monodromy transformation corresponding to the simple loop associated with $C_j(\eta)$. According to Theorem \ref{thm:t1}, part a), the matrix of $M_j$ (resp. $M_j^{-1}$) in the basis $\beta_1,\dots,\beta_N$ is upper-triangular and the only entry depending on $q$ is in position $(j,j)$, that is, it is equal to $-q^{-2}$ (resp. $-q^2$). In addition, if the refined Dubrovin conjecture holds, then since $\langle\beta_i,\beta_j\rangle=\chi(F_i,F_j)\in \ZZ$, we get that the entries of $M_j^{\pm 1}$ belong to $\ZZ[q^2,q^{-2}]$. 

Let us point out that the full exceptional collection $(F_1,\dots,F_N)$ in Theorem \ref{thm:t2} is not the same as the  full exceptional collection $(E_1,\dots,E_N)$ in the refined Dubrovin conjecture. The reason is that the objects $E_i$ correspond to oscillatory integrals in which the integration paths are rays with direction $-\eta$ while $F_i$ correspond to reference paths going in the opposite direction $\eta$. Changing the admissible direction from $\eta$ to $-\eta$ means that one has to perform a certain sequence of mutations in order to get from one exceptional collection to the other one. It turns out that the sequence of mutations that we need is well known in the theory of derived categories, i.e., this is the same sequence used to define the {\em left Koszul dual} of an exceptional collection. The precise statement is that up to a shift by $\left[\tfrac{D}{2}\right]$ (here $[x]$ denotes the integral part of $x$) the exceptional collection $(F_1,\dots,F_N)$ is the left Koszul dual to $(E_N^\vee,\dots,E_1^\vee)$. The moral is that although there is some freedom in defining a Stokes matrix and a central connection matrix of the quantum connection, the statement of the refined Dubrovin conjecture is independent of the choices that we make. The different full exceptional collection that one might get due to the discrepancy of the definitions are related by mutations in the derived category. 
%Let us point out that Dubrovin's definition of the central connection matrix in \cite{Du1999,Du2004} is different from the one used in \cite{CDG2024}, that is, they are inverse to each other.  

Finally, let us comment on the proofs. The relation between the quantum connection and its Laplace transform was studied by many people. In particular, the relation between Stokes multipliers and the monodromy data of the Laplace transform of the quantum connection is well known thanks to the work of Balser--Jurkat--Lutz \cite{BJL1981}. There is also a recent paper by Guzzetti (see \cite{Guz2016}) who was able to remove some technical conditions from the main result in \cite{BJL1981}. Although, we do not directly use any results from \cite{BJL1981}, the ideas for all proofs come from there, except for the formula for the connection matrix (see Theorem \ref{thm:refl_cc}, part a) whose proof follows the ideas of Dubrovin (see \cite{Du2004}, Theorem 4.19). Our results should not be very surprising to the experts. Especially, the work of Galkin--Golyshev--Iritani \cite{GGI2016} and Dubrovin \cite{Du2004} contain almost all ideas and results necessary to prove Theorems \ref{thm:t1} and \ref{thm:t2}. In some sense we could have written much shorter text. Nevertheless, in order to avoid gaps in the arguments due to misquoting results, we decided to have a self contained text independent of the results in \cite{BJL1981,Du2004,GGI2016}.

\section{Twisted periods of a Frobenius manifold}

As promised in the introduction, we will formulate and prove Theorem
\ref{thm:t1} more abstractly in the settings of a semisimple
Frobenius manifold. Compared to Dubrovin's theory of twisted
periods (see \cite{Du2004}, Section 4), we introduce a fundamental
solution to the second structure connection which allows us to see
better the analogy with singularity theory.  

\subsection{Frobenius manifolds}
\label{sec:frob_manifolds}
Let $M$ be a complex manifold and $\mathcal{T}_M$ be the
sheaf of holomorphic vector fields on $M$. Suppose that $M$ is
equipped with the following structures:
\begin{enumerate}
\item[(F1)]
  A non-degenerate symmetric bilinear pairing
  \begin{equation*}
    (\,\cdot\, ,\,\cdot\,):
    \mathcal{T}_M\otimes \mathcal{T}_M\to \mathcal{O}_M.
  \end{equation*}
\item[(F2)]
 A Frobenius multiplication: commutative associative multiplication
  \begin{equation*}
    \cdot \bullet \cdot :
    \mathcal{T}_M\otimes \mathcal{T}_M\to \mathcal{T}_M,
  \end{equation*}
  such that $(v_1\bullet w, v_2) = (v_1,w\bullet v_2)$ $\forall
  v_1,v_2,w\in \mathcal{T}_M$.
\item[(F3)]
  A unit vector field: global vector field $\one\in \mathcal{T}_M(M)$, such
  that, 
  \begin{equation*}
    \one\bullet v =v,\quad \nabla^{\rm L. C.}_v \one=0,\quad \forall v\in \mathcal{T}_M,
  \end{equation*}
  where $\nabla^{\rm L. C.}$ is the Levi--Civita connection of the
  pairing $(\cdot,\cdot)$.
\item[(F4)]
  An Euler vector field: global vector field $E\in \mathcal{T}_M(M)$
  such that there exists a constant $D\in \CC$, called {\em conformal
    dimension}, and 
  \begin{equation*}
    E(v_1,v_2)-([E,v_1],v_2)-(v_1,[E,v_2]) = (2-D) (v_1,v_2)
  \end{equation*}
  for all $v_1,v_2\in \T_M$.   
\end{enumerate}
Note that the complex manifold $TM\times \CC^*$ has a structure of a
holomorphic vector bundle with base $M\times \CC^*$: the fiber over
$(t,z)\in M\times \CC^*$ is $T_tM\times \{z\}\cong T_tM$ which has a
natural structure of a vector space. Given the data (F1)-(F4), we
define the so called {\em Dubrovin  connection} on the vector bundle
$TM\times \mathbb{C}^*$  
\begin{align}
  \notag
  \nabla_v & :=
             \nabla^{\rm L.C.}_v -z^{-1} v\bullet,\quad v\in
             \mathcal{T}_M,
  \\
  \notag
\nabla_{\partial/\partial z}  & :=  \frac{\partial}{\partial z} -
  z^{-1}\theta + z^{-2} E\bullet,  
\end{align}
where $z$ is the standard coordinate on
$\mathbb{C}^*=\mathbb{C}\setminus{\{0\}}$, where $v\bullet $ is an
endomorphism of $\mathcal{T}_M$ defined by the Frobenius  
multiplication by the vector field $v$, and  where $\theta:\mathcal{T}_M\to
\mathcal{T}_M$ is an $\mathcal{O}_M$-modules morphism defined by
\begin{equation*}
  \theta(v):=\nabla^{\rm L.C.}_v(E)-\Big(1-\frac{D}{2}\Big) v.
\end{equation*}

\begin{definition}\label{def:frob-manifold}
  The data $((\cdot,\cdot), \bullet, \one, E)$, satisfying the properties
  $(F1)-(F4)$, is said to be a {\em Frobenius structure} of conformal
  dimension $D$ if the corresponding Dubrovin connection is flat, that
  is, if $(t_1,\dots,t_N)$ are holomorphic local coordinates on $M$,
  then the set of $N+1$ differential operators
$
\nabla_{\partial/\partial t_i} \ (1\leq i\leq N)$,
$
\nabla_{\partial/\partial z}
$
pairwise commute. 
\qed
\end{definition}
Near $z=\infty$ the Dubrovin connection has a fundamental solution of the following form:
\beq\label{fund_sol_infty}
X(t,z)=S(t,z) z^\delta z^{-\rho},
\eeq
where $\delta$ is a diagonalizable operator, $\rho$ is a nilpotent operator, and the operator-valued series $S(t,z)=1+S_1(t)z^{-1}+\cdots$, $S_k\in \operatorname{End}(\mathcal{T}_M)$ satisfies the symplectic condition $S(t,z)S(t,-z)^T=1$, where ${}^T$ is transposition with respect to the Frobenius pairing. It can be proved that $\delta$ coincides with the semisimple part $\theta_s$ of the grading operator $\theta$ in the Jordan--Chevalley decomposition $\theta=\theta_s+\theta_n$, where the operators $\theta_s$ and $\theta_n$ are uniquely determined by the following 3 conditions:
\begin{enumerate}
  \item[(i)] Commutativity: 
    $[\theta_s,\theta_n]=0$.
  \item[(ii)] The operator $\theta_s$ is diagonalizable.
  \item[(iii)]
    The operator $\theta_n$ is nilpotent.
\end{enumerate}
Moreover, the operator $\rho=-\theta_n+\sum_{l=1}^\infty \rho_l$ where $\rho_l\neq 0$ for finitely many $l$ and $[\delta,\rho_l]=-l\rho_l$. For more details we refer to \cite{MilSa}, Section 1.3.1. Following Givental \cite{Giv2001} we will refer to the pair $(S(t,z),\rho)$ as {\em calibration} of $M$. Sometimes we will drop $\rho$ from the notation and say that $S(t,z)$ is the calibration. 
\begin{remark}
The pair $(S(t,z),\rho)$ is not uniquely determined from the Frobenius structure and in general there is no canonical choice. More precisely, one can prove that there exists a unipotent Lie group acting faithfully and transitively on the set of such pairs -- see \cite{MilSa}, Section 1.3.1.   \qed
\end{remark}
\begin{definition}\label{def:hgo}
Let $S(t,z)$ be a calibration of $M$ and $\rho$ be the corresponding nilpotent operator. 
The grading operator $\theta$ is said to be a {\em Hodge grading operator} for the calibration $(S(t,z),\rho)$ if
\begin{enumerate}
\item[(i)]
  The operator $\theta$ is diagonalizable.
\item[(ii)]
  The following commutation relation holds: $[\theta,\rho]=-\rho. $\qed
\end{enumerate}
\end{definition}
Note that if $\theta$ is a Hodge grading operator, then $\delta=\theta$ and $\rho_l=0$ for all $l\neq 1$. The fundamental solution takes the form $X(t,z)=S(t,z) z^\theta z^{-\rho}$. 
From now on we will consider only Frobenius manifolds with a fixed calibration such that $\theta$ is a Hodge grading operator. The problem that we will be interested in is local, so we will further assume that $M$ has a global flat coordinate system $(t_1,\dots,t_N)$.

Let us fix a base point $t^\circ\in M$. Put $\phi_i:=\partial/\partial t_i|_{t^\circ}$ $(1\leq i\leq N)$, then  
$\{\phi_i\}_{i=1}^N$ is a basis of the reference tangent space
$H:=T_{t^\circ}M$. The flat vector fields $\partial/\partial t_i$
($1\leq i\leq N$) provide a
trivialization of the tangent bundle $TM\cong M\times H$. This allows us to
identify the Frobenius multiplication $\bullet$ with a family of
associative  commutative multiplications $\bullet_t :H\otimes H\to H$
depending analytically on $t\in M$. The
operator $\theta:\T_M\to \T_M$ defined above preserves the subspace of
flat vector fields. It induces a linear 
operator on $H$, known to be skew symmetric with respect to
the Frobenius pairing $(\ ,\ )$.  

There are two flat connections that one can associate with the
Frobenius structure. The first one is the {\em Dubrovin
  connection} -- defined above. The Dubrovin connection in flat
coordinates takes the following form:
\ben
\nabla_{\partial/\partial t_i} & = &  \frac{\partial}{\partial t_i} -
z^{-1} \phi_i\bullet\ , \\
\nabla_{\partial/\partial z} & = & \frac{\partial}{\partial z} -z^{-1}
\theta +z^{-2} E\bullet\ ,
\een
where $z$ is the standard coordinate on $\CC^*=\CC-\{0\}$ and
for $v\in \Gamma(M,\T_M)$ we denote by $v\bullet :H\to H$ the linear
operator of Frobenius multiplication by $v$.  

We will be interested also in the {\em second structure connection }
\begin{align}
  \label{2nd_str_conn:1}
\nabla^{(m)}_{\partial/\partial t_i} & = 
\frac{\partial}{\partial t_i} + (\lambda-E\bullet_t)^{-1} (\phi_i
                                           \bullet_t) (\theta-m-1/2)\ ,
  \\
  \label{2nd_str_conn:2}
\nabla^{(m)}_{\partial/\partial\lambda} & = 
\frac{\partial}{\partial \lambda}-(\lambda-E\bullet_t)^{-1}
                                          (\theta-m-1/2)\ ,
\end{align}
where $m\in \CC$ is a complex parameter. 
This is a connection on the trivial bundle
\ben
(M\times \CC)'\times H \to (M\times \CC)',
\een
where 
\ben
(M\times \CC)'=\{ (t,\lambda)\ |\ \det (\lambda-E\bullet_t)\neq 0\}.
\een
The hypersurface $\det (\lambda-E\bullet_t)=0$ in $M\times \CC$ is
called the {\em discriminant}.

\subsection{Twisted periods}
Let $m\in \CC$ be any complex number.
\begin{definition}\label{def:tpv}
By a {\em $m$-twisted period} of the Frobenius manifold $M$  we mean a sequence $I^{(m+k)}$ $(k\in \ZZ)$ satisfying the following two properties:
\begin{enumerate}
\item[(i)]
  Flatness:
  $I^{(m+k)}$ is a horizontal section for $\nabla^{(m+k)}$.
\item[(ii)]
  Translation invariance:   $\partial_\lambda I^{(m+k)}=I^{(m+k+1)}$.\qed
\end{enumerate}
\end{definition}
The set of all $m$-twisted periods has a natural structure of a vector space. Note that if $k>0$ is sufficiently large, then the $m$-twisted period sequence is uniquely determined from $I^{(m-k)}$ only. Indeed, by translation invariance, we have
$I^{(m-k+i)}=\partial_\lambda^i I^{(m-k)}$ for all $i\geq 0$. Using \eqref{2nd_str_conn:2}
\ben
(\lambda-E\bullet)  I^{(m-k)} = (\theta-m+k-1/2) I^{(m-k-1)}.
 \een
 We get that as long as $\theta-m+k-\tfrac{1}{2}$ is invertible we can express $I^{(m-k-1)}$ in terms of $I^{(m-k)}$. Let us choose $k$ so large that $m-k+\tfrac{1}{2}$ is smaller than the real parts of all eigenvalues of $\theta$. Then, it is clear that all $I^{(m-k-1)}, I^{(m-k-2)},\dots$ can be expressed in terms of $I^{(m-k)}$.

 Suppose now that $(S(t,z),\rho)$ is a calibration. We will construct an isomorphism between $H$ and the space of all $m$-twisted periods. Let us fix a reference point $(t^\circ,\lambda^\circ)\in (M\times
\CC)'$ such that $\lambda^\circ$ is a sufficiently
large positive real number. Using the differential equations for the calibration we get that the following function is a solution to the second
structure connection $\nabla^{(m)}$ (see Proposition 3.3 in \cite{MilSa})
\beq\label{fundamental_period}
I^{(m)}(t,\lambda) = \sum_{k=0}^\infty (-1)^k S_k(t) \widetilde{I}^{(m+k)}(\lambda),
\eeq
where
\beq\label{calibrated_period}
\widetilde{I}^{(m)}(\lambda) = e^{-\rho \partial_\lambda \partial_m}
\Big(
\frac{\lambda^{\theta-m-\frac{1}{2}} }{ \Gamma(\theta-m+\frac{1}{2}) }
\Big).
\eeq
Note that both $I^{(m)}(t,\lambda)$ and $\widetilde{I}^{(m)}(\lambda)$
take values in $\op{End}(H)$. 
The second structure connection has a
Fuchsian singularity at infinity, therefore the series $I^{(m)}(t,\lambda)$ is
convergent for all $(t,\lambda)$ sufficiently close to
$(t^\circ,\lambda^\circ)$. Using the differential equations
\eqref{2nd_str_conn:1}--\eqref{2nd_str_conn:2} we extend
$I^{(m)}$ to a multi-valued analytic function on $(M\times \CC)'$
taking values in $\op{End}(H)$. We define the following multi-valued
functions taking values in $H$: 
\beq\label{H-pv}
I^{(m)}_a(t,\lambda):=I^{(m)}(t,\lambda)\, a, 
\quad a\in H,\quad m\in \CC.
\eeq 
Clearly, for each fixed $a\in H$, the sequence $I^{(m+k)}_a(t,\lambda)$
($k\in \ZZ$) is a period vector in the sense of Definition
\ref{def:tpv}. Moreover, if $k\in \ZZ$ is sufficiently negative, then
$I^{(m+k)}(t,\lambda)$ is an invertible operator. Therefore, all $m$-twisted period
vectors of $M$ have the form $I_a^{(m+k)}(t,\lambda)$ $(k\in \ZZ)$ for some $a\in H$.
Note that the analytic continuation along a closed loop around the discriminant leaves the space of $m$-twisted periods invariant. Therefore, for each $m\in \CC/\ZZ$ we have a monodromy representation 
\beq\label{mon-repr}
\pi_1((M\times\CC)',(t^\circ,\lambda^\circ) )\to \operatorname{GL}(H).
\eeq
\begin{remark}
  If $m\in \ZZ$, then the representation \eqref{mon-repr} defines the monodromy group of the Frobenius manifold. Put $q=e^{\pi\mathbf{i} m}$, then \eqref{mon-repr} defines a $q$-deformation of the monodromy group of the Frobenius manifold.\qed
\end{remark}

\subsection{Local monodromy}\label{sec:local_mon}
Recall that a point $t\in M$ is said to be {\em semisimple} if
there are local coordinates $(u_1,\dots,u_N)$ near $t$, called {\em
  canonical coordinates}, such that the multiplication and the Frobenius pairing take the following
form: 
\ben
\frac{\partial}{\partial u_i}\bullet \frac{\partial}{\partial u_j} =
\delta_{ij} \frac{\partial}{\partial u_j},\quad
\Big( \frac{\partial}{\partial u_i}, \frac{\partial}{\partial
  u_j}\Big)=
\frac{\delta_{ij}}{\Delta_j},
\een
where $\Delta_j\in \O_{M,t}$ ($1\leq j\leq N$) are holomorphic
functions that do not vanish at $t$. The Frobenius manifold $M$ is
said to be semisimple if it has at least one semisimple point. The
subset $\mathcal{K}\subset M$ of points that are not semisimple is
called the {\em caustic}. If $M$ is semisimple, then the caustic is
either the empty set or an analytic hypersurface.

From now on we will assume that $M$ is a semisimple Frobenius
manifold. Let us choose the base point $t^\circ$ such that
$\operatorname{Re} u_i(t^\circ)\neq \operatorname{Re} u_j(t^\circ)$
for $i\neq j$. Then $\operatorname{Re}  u_i(t)\neq
\operatorname{Re}  u_j(t)$ for $i\neq j$ for all $t$ sufficiently close to $t^\circ$. We
would like to describe the space of horizontal sections of
$\nabla^{(m)}$ locally in a neighbourhood of $\lambda=u_i(t)$. There is
a distinguished solution which can be constructed similarly to
\eqref{fundamental_period} but by using Givental's R-matrix instead of
the calibration $S$. Let us recall the definition of Givental's
R-matrix (see \cite{Giv2001}). Let $U(t)=\operatorname{diag}(u_1(t),\dots,u_N(t))$ and let $\Psi(t)$ be
the $N\times N$ matrix whose $(a,i)$ entry is
$\Psi_{ai}(t):=\sqrt{\Delta_i}\tfrac{\partial t_a}{\partial u_i}$. In
other words, $\Psi(t)$ is the matrix of the linear isomorphism
$\CC^N\cong T_tM$, $e_i\mapsto \sqrt{\Delta_i}\partial/\partial
u_i$ with respect to the standard basis $\{e_i\}_{i=1}^N$ of $\CC^N$
and the flat basis $\{\partial/\partial t_a\}_{a=1}^N$ of
$T_tM$. According to Givental, there exists a unique operator-valued 
formal series $R(t,z)=1+R_1(t)z+R_2(t) z^2+\cdots $ such that the
Dubrovin connection has a formal asymptotic solution at $z=0$ of the
form $\Psi(t) R(t,z) e^{U(t)/z}$. Moreover, the matrix series $R(t,z)$
satisfies the symplectic condition $R(t,-z)^t R(t,z)=1$ where ${}^t$
is the standard transposition operation for matrices. Using the differential equations for $\Psi(t) R(t,z) e^{U(t)/z}$, we get that
the following Laurent series is a solution to $\nabla^{(m)}$:
\beq\label{refl_period}
I_i^{(m)}(t,\lambda):=
\sqrt{2\pi} \, \sum_{k=0}^\infty (-1)^k\, \Psi(t) R_k(t) e_i \,
\frac{(\lambda-u_i)^{k-m-1/2} }{\Gamma(k-m+1/2)}.
\eeq
The following proposition is well known (see \cite{MilSa}, Section 3.2.2)
\begin{proposition}\label{prop:local_periods}
Suppose that $m-\tfrac{1}{2}\nin \ZZ$.  In a neighbourhood of
$\lambda=u_i(t)$, the space of holomorphic solutions of
$\nabla^{(m)}$ is a subspace of co-dimension 1 in the space of all
solutions. 
\end{proposition}
From now on we will assume that $m-\tfrac{1}{2}\nin \ZZ$. Under this
condition every solution to $\nabla^{(m)}$, locally near
$\lambda=u_i(t)$ is a sum of a holomorphic solution and $c
I^{(m)}_i(t,\lambda)$ for some constant $c$. In particular, we can
easily describe the local monodromy of $\nabla^{(m)}$ near
$\lambda=u_i(t)$. The analytic continuation along a simple
counter-clockwise loop around
$\lambda=u_i(t)$ transforms $I_i^{(m)}(t,\lambda)\mapsto -q^{-2}
I_i^{(m)}(t,\lambda)$ where $q:=e^{\pi\mathbf{i} m}$. Note that
locally near $\lambda=u_i(t)$ the holomorphic
solutions are precisely the monodromy invariant ones. 
\begin{remark}
In the case when $m-\tfrac{1}{2}\in \ZZ$ there might be solutions
involving $\log (\lambda-u_i(t))$.  
\end{remark}

Let us introduce the following pairings $h_m: H\times H\to \CC$ 
\beq\label{herm_pairing}
h_m(a,b):= (I^{(m)}_a(t,\lambda),(\lambda-E\bullet) I^{(-m)}_\beta(t,\lambda)).
\eeq
Using the differential equations of $\nabla^{(\pm m)}$ it is easy to
check that $h_m(a,b)$ is independent of $t$ and $\lambda$. In
particular, it is a monodromy invariant pairing between the space of
$m$-twisted and $(-m)$-twisted periods. Note that we also have the
following symmetry:
\ben
h_m(a,b)=h_{-m}(b,a),\quad \forall a,b\in H.
\een
It turns out that there is an explicit formula
for $h_m$ in terms of the Hodge grading operator $\theta$ and the
nilpotent operator $\rho$. Let us recall
the so-called {\em Euler pairing}
\beq\label{euler_p}
\langle a, b\rangle :=\frac{1}{2\pi} ( a, e^{\pi \mathbf{i}\theta}
e^{\pi\mathbf{i}\rho} b),\quad a,b\in H.
\eeq
As a byproduct of the proof of Theorem \ref{thm:refl_cc} we will get the following simple formula: 
\ben
h_m(a,b) = q \langle a, b\rangle + q^{-1} \langle b, a\rangle, 
\een
where  $q:=e^{\pi\mathbf{i} m}$.
Given a reference path (avoiding the discriminant) from
$(t^\circ,\lambda^\circ)$ to $(t,u_i(t))$, there exists a
vector $\beta_i(m)$ such that the period vector
$I^{(m)}_{\beta_i(m)} (t,\lambda)=I^{(m)}_i(t,\lambda)$. Since the
series in \eqref{refl_period} involves fractional powers of
$\lambda-u_i$, the value of $\beta_i(m)$ depends not only on the
reference path but also on the choice of a branch for $\log
(\lambda-u_i)$. In other words, the value of $\beta_i(m)$ is unique up
to a factor in the spiral $(-q^{-2})^\ZZ$. Note that fixing the
reference path and the branch of $\log(\lambda-u_i)$ determines
$\beta_i(m)$ for all $m\in \CC\setminus{\{ \tfrac{1}{2}+\ZZ \}}$. We will refer
to $\beta_i(m)$ as the {\em $m$-twisted reflection vector}
corresponding to the reference path.  
\begin{lemma}\label{le:local_Hm}
  a) We have $h_m(\beta_i(m),\beta_i(-m)) = q+q^{-1}$.

  b) If $a\in H$ is such that $I^{(m)}_a(t,\lambda)$ is holomorphic at
  $\lambda=u_i(t)$, then $h_m(a,\beta_i(-m))=0$. 
\end{lemma}
\proof
The proof is obtained by substituting formula \eqref{refl_period} into
the definition \eqref{herm_pairing} and extracting the leading order
term in the Laurent series expansion at $\lambda=u_i$. If we do this
for the pairing in part a) we will get
\ben
\frac{2\pi}{\Gamma(-m+\tfrac{1}{2})\Gamma(m+\tfrac{1}{2})} = 2 \sin
\pi(m+\tfrac{1}{2}) = 2\cos (\pi m) = q+q^{-1}. 
\een
This proves a). The proof of b) is similar.
\qed

\begin{proposition}\label{prop:loc_mon}
  Let $\sigma_i$ be the local monodromy transformation of
  $\nabla^{(m)}$ corresponding to a simple counter-clockwise loop around
  $\lambda=u_i(t)$. Then
  \ben
  \sigma_i(a) = a- q^{-1} h_m(a,\beta_i(-m))\, \beta_i(m),\quad a\in H,
  \een
  where $\beta_i(\pm m)\in H$ are $\pm m$-twisted reflection vectors
  corresponding to the simple loop. 
\end{proposition}
\proof
According to Proposition \ref{prop:local_periods} there exists a
decomposition $a= a'+ k \beta_i(m)$ such that
$I^{(m)}_{a'}(t,\lambda)$ is analytic at $\lambda=u_i(t)$. We have
\ben
\sigma_i(a) = a'-k q^{-2} \beta_i(m) = a-(1+q^{-2}) k \beta_i(m). 
\een
On the other hand, recalling Lemma \ref{le:local_Hm} we have
$h_m(a,\beta_i(-m)) = k (q+q^{-1})= k(1+q^{-2}) q$. Therefore,
$k(1+q^{-2})= q^{-1} h_m(a,\beta_i(-m)) $ which yields the formula
that we had to prove.
\qed

\subsection{Asymptotic expansions and Stokes matrices}\label{sec:AE}
Let us recall the settings from Section \ref{sec:mdrv}. 
\begin{figure}[t]%% placement specifier
%% Use \includegraphics command to insert graphic files. Place graphics files in 
%% working directory.
\centering%% For centre alignment of image. 
\includegraphics{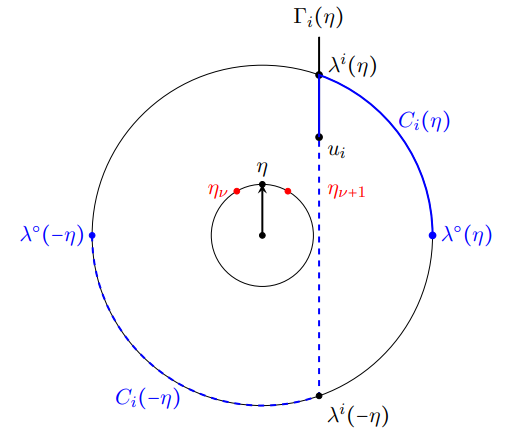}
%% Use \caption command for figure caption and label.
\caption{Reference paths and admissible directions}\label{fig:rp}
\end{figure}
Let $\mathbb{S}^1=\{\eta\in \CC\ |\ |\eta|=1\}$ be the unit circle. Note that a point $\eta\in \mathbb{S}^1$ is an admissible direction iff the ray $\Gamma_i(\eta):=\{u_i+\eta s\ |\ s\in \mathbb{R}_{\geq 0}\} $ does not pass through $u_j$ for $j\neq i$. 
A direction which is not admissible is said to be {\em critical}. If $\eta\in \mathbb{S}^1$ is a critical direction, then $-\eta$ is also critical. Therefore, the number of all critical directions is even, say $2\mu$ for some $\mu\in \ZZ$. Following \cite{BJL1981} we order the critical directions in a clockwise order $\eta_0,\eta_1,\dots,\eta_{2\mu-1}$ in such a way that 
\ben
-\tfrac{\pi}{2}<
\operatorname{Arg}(\eta_{2\mu-1})<\cdots <
\operatorname{Arg}(\eta_0)\leq \tfrac{3\pi}{2}.
\een
Let us assume that $\lambda^\circ >|u_i|$ for all $i$. By definition, our choice of the reference point $t^\circ$ is such that $\mathbf{i}=\sqrt{-1}$ is an admissible direction. 
This is going to be our default admissible direction. It will be
convenient to introduce an auxiliary reference point
$\lambda^\circ(\eta):=-\mathbf{i}\eta \lambda^\circ$. Note that if we
continuously change the admissible direction from $\ii$ to $\eta$, then we will obtain a path connecting $\lambda^\circ=\lambda^\circ(\ii)$ and $\lambda^\circ(\eta)$. 

Suppose that $\eta$ is an admissible direction. 
Let us consider the following oscillatory integrals:
\beq\label{osc_integral_i}
X_i(\eta,t,z)=\frac{1}{\sqrt{2\pi}}\,
(-z)^{m-1/2} 
\int_{\Gamma_i(\eta)} 
e^{\lambda/z} I_i^{(m)}(t,\lambda)d\lambda,
\eeq
where $m\in \CC$ is a complex number such that $\operatorname{Re}(m)<0$. The integral is absolutely convergent for all $z$ in the half-plane 
\ben
H_\eta:= \{z\in \CC\ |\ 
\operatorname{Re}(\eta/z)<0\}=
\{z\in \CC\ |\  \operatorname{Re}(\eta\, \overline{z})<0\},
\een
where $\overline{z}$ is the complex conjugate of $z$. Note that if $z_i=x_i +\mathbf{i} y_i$ ($i=1,2$), then $\operatorname{Re}(z_1\,\overline{z}_2)=x_1 x_2+ y_1 y_2$ is the standard Euclidean pairing. Therefore, $H_\eta$ is the half-plane in $\CC$ whose boundary is the line orthogonal to $\eta$ and which does not contain $\eta$. The definition of \eqref{osc_integral_i} requires also a choice of $\log (-z)$ and $\log (\lambda-u_i)$ in order to be able to define fractional powers of $-z$ and $(\lambda-u_i)$. Recall that the admissible direction $\eta$ is equipped with a choice of $\log \eta$ such that \eqref{adm_Arg} holds. Since $\eta\in H_{-\eta}$ the choice of $\log$ at $\eta$ uniquely determines a holomorphic branch of $\log$ defined on the entire half-plane $H_{-\eta}$. Note that when $\lambda\in \Gamma_i(\eta)$ we have $\lambda-u_i\in H_{-\eta}$. Moreover, for the convergence of the integral \eqref{osc_integral_i} we have to require that $-z\in H_{-\eta}$. Therefore, both $-z$ and $(\lambda-u_i)$ belong to $H_{-\eta}$ and we have a natural choice of the value of $\log$. We get that $X_i(\eta,t,z)$ is a holomorphic function for $z\in H_\eta$. Put $\lambda-u_i= - s z$ and note that for $z\in -\eta \mathbb{R}_{>0}$ we have 
\ben
\int_{\Gamma_i(\eta)} e^{\lambda/z} 
\frac{(\lambda-u_i)^{k-m-\tfrac{1}{2}}}{
\Gamma(k-m+\tfrac{1}{2})} d\lambda = e^{u_i/z}\, 
(-z)^{k-m+\tfrac{1}{2}}. 
\een
Recalling the stationary phase asymptotic method and using the expansion \eqref{refl_period} we get that
\beq\label{asympt:Yi}
X_i(\eta,t,z)\ \sim\ 
\Psi(t) R(t,z) e_i e^{u_i/z},\quad z\to 0 \mbox{ in }
H_{\eta}. 
\eeq
Suppose that $\eta'$ and $\eta''$ are two admissible directions, such
that, $\eta'$ and $\eta''$ belong to the same clockwise arc from
$\eta_\nu$ to $\eta_{\nu+1}$, i.e., the arc bounded by two adjacent
critical directions. By definition, the sector between the rays
$\Gamma_i(\eta')$ and $\Gamma_i(\eta'')$ does not contain $u_j$ for
$j\neq i$. This implies that $I^{(m)}_i(t,\lambda)$ extends to a
holomorphic function in that sector. Using the Cauchy residue theorem (cf. Proposition \ref{prop:fund_periods}, where the same technique is used to prove part c), we get that 
$X_i(\eta',t,z)=X_i(\eta'',t,z)$ for all
$z\in H_{\eta'}\cap H_{\eta''}$. In other words, for every admissible
direction $\eta$, the oscillatory integral \eqref{osc_integral_i}
extends analytically in $z$ for all $z\in H_{\eta_\nu}\cup
H_{\eta_{\nu+1}}$ where $\eta_\nu$ and $\eta_{\nu+1}$ are the two
critical directions adjacent to $\eta$. 
\begin{figure}[t]%% placement specifier
%% Use \includegraphics command to insert graphic files. Place graphics files in 
%% working directory.
\centering%% For centre alignment of image.
\includegraphics{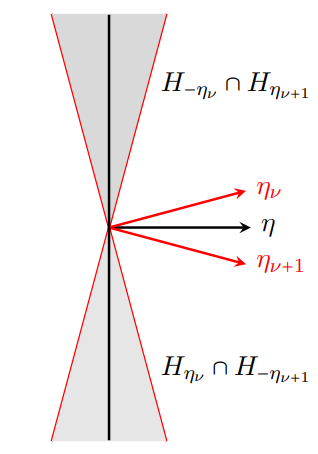}
%% Use \caption command for figure caption and label.
\caption{Domains of analyticity in $z$}\label{fig:half_pl}  
\end{figure}
Figure \eqref{fig:half_pl} might help visualize the domains of
analyticity.
Let us denote by $X(\eta,t,z)$ the matrix of size $N\times N$ whose
$i$-th column is $X_i(\eta,t,z)$.
Since both $X(-\eta,t,z)$
and $X(\eta,t,z)$ are solutions to the Dubrovin connection
for $z\in H_{\eta_\nu}\cap H_{-\eta_{\nu+1}}$, there exists a matrix
$V_+(\eta)$  such that
\beq\label{V+}
X(-\eta,t,z) = X(\eta,t,z) \, V_+(\eta),\quad
\forall z\in H_{\eta_\nu}\cap H_{-\eta_{\nu+1}}.
\eeq
Similarly, there exists a matrix $V_-(\eta)$ such that
\beq\label{V-}
X(-\eta,t,z) = X(\eta,t,z) \, V_-(\eta),\quad
\forall z\in H_{-\eta_\nu}\cap H_{\eta_{\nu+1}}.
\eeq
In both formulas \eqref{V+} and \eqref{V-} we define $-\eta$ by
continuously rotating $\eta$ on $180^\circ$ in {\em clockwise}
direction. 
The matrices $V_+$ and $V_-$ are called {\em Stokes matrices}. There
is a simple relation between $V_+$ and $V_-$ (see \cite{Du1996},
Proposition 3.10). 
\begin{proposition}\label{prop:stokes_relation}
We have $V_+=V_-^t$ where ${}^t$ is the usual transposition
operation of matrices. 
\end{proposition}
\proof 
Let $g_{ab}=(\partial/\partial t_a,\partial/\partial t_b)$ be the
matrix of the Frobenius pairing. Note that $\Psi^t g \Psi = 1$.
We claim that $X(-\eta,t,-z)^t g X(\eta,t,z) = 1$. First of all, using
that $X(\pm \eta,t,z)$ is a solution to the Dubrovin connection we get
that $A:=X(\eta,t,-z)^t g X(\eta,t,z)$ is a constant independent of $t$
and $z$. Let us recall the asymptotic expansions
$X(\eta,t,z)\sim \Psi(t) R(t,z) e^{U/z}$ and
$X(-\eta,t,-z)\sim \Psi(t) R(t,-z) e^{-U/z}$
where $z\to 0$ and $z\in H_\eta$.
In particular, we have that both $X(\eta,t,z)e^{-U/z}$ and
$X(-\eta,t,-z)e^{U/z}$ have limit when $z\to 0$ and $z\in H_\eta$
which must be $\Psi$ (in both cases). Therefore, $e^{U/z} A
e^{-U/z}\to \Psi^t g \Psi=1$ when $z\to 0$ in the half-plane
$H_\eta$. This implies that the diagonal entries of $A$ must be 1. For
$i\neq j$, since $\eta$ is an admissible direction, we can find
$z_0\in H_\eta$ such that $\operatorname{Re}((u_i-u_j)/z_0)>0$. If
$A_{ij}\neq 0$, then the $(i,j)$ entry of $e^{U/z} A e^{-U/z}$, that is,
$e^{(u_i-u_j)/z}A_{ij}$ has an exponential growth as $z\to 0$ in the
direction of $z_0$ -- contradiction. This completes the proof of our
claim. 

Suppose that
$z\in H_{\eta_\nu}\cap H_{-\eta_{\nu+1}}$. Then we have
$X(\eta,t,z)= X(-\eta,t,z) V_+^{-1}$ and
$X(-\eta,t,-z)= X(\eta,t,-z) V_-$ because
$-z\in H_{-\eta_\nu}\cap H_{\eta_{\nu+1}}$. We get
\ben
1=X(-\eta,t,-z)^t g X(\eta,t,z) =
V_-^t\, 
X(\eta,t,-z)^t g X(-\eta,t,z) V_+^{-1} =
V_-^t\,  V_+^{-1}. \qed
\een
The following proposition is well known (see \cite{Du1996},
Proposition 3.10).
\begin{proposition}\label{prop:stokes_triangle}
Let $V_{+,ij}$ be the $(i,j)$-entry of the Stokes matrix
$V_+$. Suppose that $z\in H_{\eta_\nu}\cap H_{-\eta_{\nu+1}}$. Then
\begin{enumerate}
  \item[(a)] The diagonal entries $V_{+,ii}=1$ for all $1\leq i\leq
    N$.
  \item[(b)]
    If $\operatorname{Re} ( (u_i-u_j)/z )>0$, then
    $V_{+,ij}=0$.  
\end{enumerate}
\end{proposition}
\proof
Using the asymptotic expansion \eqref{asympt:Yi} and the identity
\eqref{V+} we get that $e^{U/(sz)} V_+ e^{-U/(sz)}$, where $s\in
\RR_{>0}$ and $z\in H_{\eta_\nu}\cap H_{-\eta_{\nu+1}}$,  has a limit when
$s\to 0$ which must be the identity matrix. Part (a) follows
immediately from this observation. For part (b) we need only to notice
that if $\operatorname{Re} ( (u_i-u_j)/z )>0$, then $e^{(u_i-u_j)/(sz)}$
has an exponential growth as $s\to 0$. Therefore, the limit of $e^{(u_i-u_j)/(sz)}
V_{+,ij}$ exists only if $V_{+,ij}=0$. \qed
\begin{remark}
Note that the condition $\operatorname{Re} ( (u_i-u_j)/z )>0$ in (b)
is independent of the choice of $z\in H_{\eta_\nu}\cap
H_{-\eta_{\nu+1}}$, otherwise the direction of $u_i-u_j$ or $u_j-u_i$
must belong to the cone spanned by $\eta_\nu$ and $\eta_{\nu+1}$
contradicting the fact that there are no critical directions between
$\eta_\nu$ and $\eta_{\nu+1}$.
\qed
\end{remark}
Recalling Proposition \ref{prop:stokes_relation} we get the following corollary.
\begin{corollary}\label{prop:stokes_triangle_lower}
Let $V_{-,ij}$ be the $(i,j)$-entry of the Stokes matrix
$V_-$. Suppose that $z\in H_{-\eta_\nu}\cap H_{\eta_{\nu+1}}$. Then
\begin{enumerate}
  \item[(a)] The diagonal entries $V_{-,ii}=1$ for all $1\leq i\leq
    N$.
  \item[(b)]
    If $\operatorname{Re} ( (u_i-u_j)/z )>0$, then
    $V_{-,ij}=0$.
    \qed
\end{enumerate}
\end{corollary}

\subsection{Stokes matrices and the intersection pairing}
Let us recall the paths $C_i(\eta)$ ($1\leq i\leq N$) introduced in Section \ref{sec:mdrv}. Let $\beta_i(m)\in H$ be the reflection 
vector corresponding to the path $C_i(\eta)$, that is,
$I^{(m)}_i(t,\lambda) =I^{(m)}_{\beta_i(m)}(t,\lambda)$.
We would like to express the entries of the Stokes matrix $V_+$ in terms of the Euler pairing and the reflection vectors $\beta_i(m)$. 
%Let $\lambda^i(\eta)$
%be the intersection of the ray $\Gamma_i(\eta)$ and the circle
%$|\lambda|=\lambda^\circ$ (see Figure \ref{fig:rp}).  
%Recall that we fixed an  auxiliary reference point $\lambda^\circ(\eta)=-\ii\eta
%\lambda^\circ$ which is connected to $\lambda^\circ$ by continuously
%deforming the direction from $\ii$ to $\eta$. We define the path
%$C_i(\eta)$ to be the composition of the counter-clockwise oriented
%arc from $\lambda^\circ(\eta)$ to $\lambda^i(\eta)$ and the line segment from $\lambda^i(\eta)$ to $u_i$ (see the blue paths on Figure
%\ref{fig:rp}). 
Following
\cite{BJL1981}, let us introduce vectors $\beta_i^*(m)$ ($1\leq i\leq
N$) such that $h_m(\beta_i^*(m),\beta_j(-m))=\delta_{ij}$. The
pairing $h_m$ is non-degenerate except for finitely many $m\in
\CC$. The definition of $\beta_i^*(m)$ makes sense except for finitely
many values of $m$. The properties of the corresponding period vectors
can be summarized as follows (compare with \cite{BJL1981}, Proposition
  1 and Theorem 2').
\begin{proposition}\label{prop:fund_periods}
  a) The period vector $I^{(m)}_{\beta_i^*(m)}(t,\lambda)$ is analytic
  at $\lambda=u_j$ for $j\neq i$, where the value of the period is
  specified via the reference path $C_j(\eta)$.

  b) The following formula holds:
  \ben
  \beta_i(m) = (q+q^{-1})\, \beta_i^*(m) + \sum_{j:j\neq i}
  h_m(\beta_i(m),\beta_j(-m)) \beta_j^*(m). 
  \een

  c) Let $\gamma_i(-\eta)\subset \CC$ be a contour starting at
  $\lambda=\infty$, approaching $\lambda=u_i$ along the ray
  $\Gamma_i(-\eta)$, making a small loop around $\lambda=u_i$, and
  finally returning back to $\lambda=\infty$ along $\Gamma_i(-\eta)$.
  Put
  \ben
  X_i^*(-\eta,t,z) = - q \,
  \frac{(-z)^{m-1/2}}{\sqrt{2\pi}} \,
  \int_{\gamma_i(-\eta)} e^{\lambda/z} I^{(m)}_{\beta_i^*} (t,\lambda) d\lambda,
  \een
  where the value of $I^{(m)}_{\beta_i^*} (t,\lambda)$ is determined
  by the reference path $C_i(\eta)$ as follows: first we fix the value
  at the intersection of $\gamma_i(-\eta)$ and $C_i(\eta)$, then we
  extend by continuity to the remaining points of $\gamma_i(-\eta)$. 
  Then $X_i^*(-\eta,t,z)$ coincides with $X_i(-\eta,t,z)$ for all
  $z\in H_{-\eta}$ where in order to specify the value of
  $X_i(-\eta,t,z)$  and of $\log (-z)$ we take a clockwise rotation from $\eta$ to
  $-\eta$. 
\end{proposition}
\proof
a) Using Proposition \ref{prop:loc_mon} we get that
$\sigma_j(\beta_i^*) = \beta_i^* - q^{-1} 
h_m(\beta_i^*(m),\beta_j(-m)) \, \beta_j(m) = \beta_i^*$ where we use
that for $i\neq j$ the pairing
$h_m(\beta_i^*(m),\beta_j(-m))=0$. Therefore, the period
$I^{(m)}_{\beta_i^*(m)}(t,\lambda)$ is single-valued in a neighborhood
of $\lambda=u_j$ which is possible only if it is holomorphic.

b) This is obvious from the definition of $\beta_i^*(m)$.

c) According to parts a) and b), the difference
$
(q+q^{-1}) I^{(m)}_{\beta_i^*} (t,\lambda)-
I^{(m)}_{\beta_i} (t,\lambda)$
is holomorphic at $\lambda=u_i$. Moreover, the periods being solutions
to a Fuchsian differential equation, have at most polynomial growth at
$\lambda=\infty$. Recalling the Cauchy residue theorem we get
\beq\label{gamma-1}
\int_{\gamma_i(-\eta)}
e^{\lambda/z}
( (q+q^{-1})\, 
I^{(m)}_{\beta_i^*} (t,\lambda)-I^{(m)}_{\beta_i}
(t,\lambda))d\lambda = 0. 
\eeq
Using integration by parts, it is easy to check that $X_i^*$ is
invariant under the shift $m\mapsto m+1$. Therefore, we may assume
that $\operatorname{Re}(m)<0$. Note that
\beq\label{gamma-2}
\int_{\gamma_i(-\eta)}
e^{\lambda/z} I^{(m)}_{\beta_i} (t,\lambda) d\lambda
=
\Big(-\int_{\Gamma_i(-\eta)}-q^{-2} \int_{\Gamma_i(-\eta)}\Big)
e^{\lambda/z} I^{(m)}_{\beta_i} (t,\lambda) d\lambda. 
\eeq
Indeed, let us split the integration contour $\gamma_i(-\eta)$ into 3 pieces:
$-\Gamma_i(-\eta)\setminus{[u_i,u_i-\epsilon \eta]}$ going from
$-\infty\eta$ to $u_i-\epsilon\eta$, an
$\epsilon$-loop around $\lambda=u_i$ starting and ending at
$u_i-\epsilon\eta$, and the ray
$ \Gamma_i(-\eta)\setminus{[u_i,u_i-\epsilon \eta]}$ from
$u_i-\epsilon\eta$ to $-\infty\eta$ . Recalling the
Laurent series expansion \eqref{refl_period}, we get that under the
analytic continuation along the $\epsilon$-loop, the integrand
$I^{(m)}_{\beta_i} (t,\lambda) $ gains a factor of
$-q^{-2}$. Since the orientations of the first and the third contours
are opposite, we get that the corresponding integrals differ by a
factor of $q^{-2}$. Furthermore, since $m<0$, the integral along the loop has a
contribution which vanishes in the limit $\epsilon\to 0$. This
completes the proof of formula \eqref{gamma-2}. Using formulas
\eqref{gamma-1} and \eqref{gamma-2} we get
\ben
\int_{\gamma_i(-\eta)}
e^{\lambda/z}\,
I^{(m)}_{\beta_i^*} (t,\lambda) =
-\frac{1+q^{-2}}{q+q^{-1}}\,
\int_{\Gamma_i(-\eta)}
e^{\lambda/z}\,
I^{(m)}_{\beta_i}
(t,\lambda))d\lambda = -q^{-1}
\int_{\Gamma_i(-\eta)}
e^{\lambda/z}\,
I^{(m)}_{\beta_i}
(t,\lambda))d\lambda.
\een
The statement in part c) follows from the above formula.
\qed

Suppose now that $\widetilde{\eta}$ is another admissible direction
such that $\eta_\nu < \widetilde{\eta} <\eta_{\nu-1}$. In other
words, $\widetilde{\eta}$ is obtained from $\eta$ by crossing the
critical direction $\eta_\nu$. Let $\widetilde{\beta}_i$ and
$\widetilde{\beta}_i^*$ be the vectors corresponding to the reference
paths $C_i(\widetilde{\eta})$. We would like to express $\widetilde{\beta}_i$ and
$\widetilde{\beta}_i^*$ in terms of $\beta_i$ and $\beta_i^*$. Let us
split the points $u_1,\dots,u_N$ into groups such that each group
belongs to a ray with direction $\eta_\nu$ and the rays of different
groups are different. Let $(u_{j_1},\dots,u_{j_k})$  be one such group
whose elements are ordered in such a way that $u_{j_a} = u_{j_k}+s_a
\eta_\nu$ for some real numbers $s_1>s_2>\cdots >s_k=0$. We will refer
to such a sequence $(u_{j_1},\dots,u_{j_k})$ as {\em
  $\eta_\nu$-sequence}. Clearly this splitting is uniquely determined
by the critical direction $\eta_\nu$.
\begin{figure}[t]%% placement specifier
%% Use \includegraphics command to insert graphic files. Place graphics files in 
%% working directory.
\centering%% For centre alignment of image.
\includegraphics{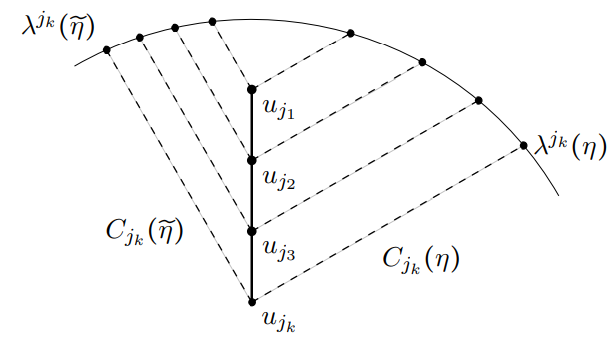}
%% Use \caption command for figure caption and label.
\caption{$\eta_\nu$-sequence}\label{fig:eta_sequence}
\end{figure}
\begin{proposition}\label{prop:beta-wallcrossing}
  Suppose that $(u_{j_1},\dots,u_{j_k})$ is a
  $\eta_\nu$-sequence. Then
  \ben
  \widetilde{\beta}_{j_1} =\beta_{j_1},\quad 
  \widetilde{\beta}_{j_t} =\sigma^{-1}_{j_1}\cdots \sigma^{-1}_{j_{t-1}}
  (\beta_{j_t}) \quad
  (2\leq t\leq k)
  \een
  and
  \ben
  \widetilde{\beta}_{j_k}^* = \beta_{j_k}^*,\quad
  \widetilde{\beta}_{j_t}^* = \beta_{j_t}^* +
  \sum_{a=t+1}^k q^{-1}\, h_m(\beta_{j_t}(m),\beta_{j_a}(-m))
  \beta_{j_a}^*\quad (1\leq t\leq k-1).  
  \een
\end{proposition}
\proof
For part a), let us look at figure \eqref{fig:eta_sequence}. By
definition, the period $I^{(m)}_{j_k}(t,\lambda)$ defined in a
neighbourhood of $\lambda=u_{j_k}$ is obtained from
$I^{(m)}_{\beta_{j_k}}(t,\lambda)$ via the analytic
continuation along the reference path
$C_{j_k}(\eta)$. On the other hand, the analytic
continuation of $I^{(m)}_{j_k}(t,\lambda)$ along the inverse of the reference path
$C_{j_k}(\widetilde{\eta})$ yields
$I^{(m)}_{\widetilde{\beta}_{j_k}}(t,\lambda)$.   The conclusion is
that the cycle 
$\widetilde{\beta}_{j_k}$ is obtained from $\beta_{j_k}$ after a
monodromy transformation along a small modification of the loop
consisting of the following 3 pieces:  the line segment 
from $\lambda^{j_k}(\eta)$ to $u_{j_k}$, the line segment from
$u_{j_k}$ to $\lambda^{j_k}(\widetilde{\eta})$, and the arc from
$\lambda^{j_k}(\widetilde{\eta})$ to $\lambda^{j_k}(\eta)$.  The small
modification, necessary to avoid the singularity $u_{j_k}$, is taken
as follows: 
when we approach $u_{j_k}$ along $C_{j_k}(\eta)$ we have
to stop slightly before hitting $u_{j_k}$, make an anti-clockwise
rotation along $u_{j_k}$ until we hit $C_{j_k}(\widetilde{\eta})$ and
then continue along the old contour. The reason why we have to make
anti-clockwise rotation, and not clockwise, is that the value of
$\log(\lambda-u_{j_k})$, needed to define $I^{(m)}_{j_k}(t,\lambda)$,
is determined by $\log \eta$ (resp. $\log 
\widetilde{\eta}$)  when $\lambda\in C_{j_k}(\eta)$ (resp. $\lambda\in
C_{j_k}(\widetilde{\eta})$). Since $\widetilde{\eta}$ is obtained from
$\eta$ by anti-clockwise rotation, we have to go around $u_{j_k}$
anti-clockwise. Clearly, the loop
decomposes into simple loops going successively {\em clockwise} around the  points
$u_{j_{k-1}}, \dots, u_{j_1}$ in the given order, i.e., first around
$u_{j_{k-1}}$, then $u_{j_{k-2}}$, etc., finally $u_{j_1}$. After this
discussion the first formula that we have to prove should be clear.

For the second formula, let us argue by induction. The fact that
$\widetilde{\beta}_{j_k}^*=\beta_{j_k}^*$ follows immediately from the
first part of the proposition which implies that $\widetilde{\beta}_{j_t}$ is a
sum of $\beta_{j_t}$ and a linear combination of
$\beta_{j_1},\dots,\beta_{j_{t-1}}$. Suppose that the formula is
proved for all $t>s$. Let us find coefficients $b_{s+1},\dots, b_{k}$
such that
\ben
B:=\beta_{j_s}^* + b_{s+1} \beta_{j_{s+1}}^* +\cdots + b_k \beta_{j_k}^*
\een
satisfies the defining equations of $\widetilde{\beta}^*_{j_s}$. Note
that $h_m(B,\widetilde{\beta}_{j_a} ) = \delta_{j_s,j_a}$ for all
$a\leq s$. Therefore, we have to solve the equations
$h_m(B,\widetilde{\beta}_{j_a})=0$ ($a=s+1,\dots,k$) for
$b_{s+1},\dots,b_k$. For $a=s+1$, we get
\beq\label{eq_s+1}
h_m( B,
\widetilde{\beta}_{j_{s+1}} ) = h_m(\sigma_{j_s}\cdots
\sigma_{j_1}(B), \beta_{j_{s+1}}) =
h_m(\sigma_{j_s} (\beta_{j_s}^*) + b_{s+1} \beta_{j_{s+1}}^*,
\beta_{j_{s+1}}),
\eeq
where we used that the pairing $h_m$ is monodromy invariant and we
dropped from $B$ all terms that do not contribute. Note that by
Proposition \ref{prop:loc_mon} we have
\ben
\sigma_{j_s}
(\beta_{j_s}^*) = \beta_{j_s}^* - q^{-1} h_m(\beta_{j_s}^*,\beta_{j_s})
\beta_{j_s} = \beta_{j_s}^*-q^{-1} \,  \beta_{j_s}. 
\een
Substituting this formula in \eqref{eq_s+1} we get $b_{s+1}=q^{-1}
h_m(\beta_{j_s},\beta_{j_{s+1}})$. Suppose that we proved that
$b_{s+i} = q^{-1} \, h_m(\beta_{j_s},\beta_{j_{s+i}})$ for
$i=1,\dots,l$. In order to determine $b_{s+l+1}$, let us consider
the equation $h_m(B,\widetilde{\beta}_{j_{s+l+1}})=0$. We get
\ben
h_m(\sigma_{j_{s+l}} \cdots \sigma_{j_s}
(B),\beta_{j_{s+l+1}}) = 0.
\een
Recalling the ansatz for $B$ we get that in
$\sigma_{j_{s+l}} \cdots \sigma_{j_s} (B) $ only the
following terms will contribute:
\ben
\sigma_{j_{s+l}} \cdots \sigma_{j_{s+1} }(\sigma_{j_s
}(\beta_{j_s}^*)) +
b_{s+1} \sigma_{j_{s+l}} \cdots \sigma_{j_{s+2} }(\sigma_{j_{s+1}
}(\beta_{j_{s+1}}^*))+\cdots +
b_{s+l} \sigma_{j_{s+l}}(\beta_{j_{s+l}}^*) + b_{s+l+1} \beta_{j_{s+l+1}}^*. 
\een
We have $\sigma_{j_t}(\beta_{j_t}^*) = \beta_{j_t}^* - q^{-1}
\beta_{j_t}$ for $t=s,s+1,\dots, s+l$. Note that  
$\beta_{j_t}^*$ is fixed by $\sigma_{j_{t+1}},\dots, \sigma_{j_{s+l}}$
and that $h_m(\beta_{j_t}^*, \beta_{j_{s+l+1}})=0$. Therefore, we may
replace the above expression with 
\beq\label{hmB}
\sigma_{j_{s+l}} \cdots \sigma_{j_{s+1} }
(-q^{-1} \beta_{j_s} ) +
b_{s+1} \sigma_{j_{s+l}} \cdots \sigma_{j_{s+2} }
(-q^{-1} \beta_{j_{s+1}})+\cdots +
b_{s+l} (-q^{-1} \beta_{j_{s+l}}) + b_{s+l+1} \beta_{j_{s+l+1}}^*. 
\eeq
Let us add the first two terms. After pulling out the common
expression $- q^{-1}\, \sigma_{j_{s+l}} \cdots \sigma_{j_{s+2} }$
we are left with
\ben
\sigma_{j_{s+1}}(\beta_{j_s}) + b_{s+1} \beta_{j_{s+1}} =
\beta_{j_s}-q^{-1} \, h_m(\beta_{j_s},\beta_{j_{s+1}}) \beta_{j_{s+1}} +
b_{s+1} \beta_{j_{s+1}} = \beta_{j_s},
\een
where we used the formula for $b_{s+1}$. Therefore, after adding up
the first two terms in \eqref{hmB} we get
\ben
\sigma_{j_{s+l}} \cdots \sigma_{j_{s+2} }(-q^{-1} \beta_{j_s}) +
b_{s+2} \sigma_{j_{s+l}} \cdots \sigma_{j_{s+3} }
(-q^{-1} \beta_{j_{s+2}})+\cdots +
b_{s+l} (-q^{-1} \beta_{j_{s+l}}) + b_{s+l+1} \beta_{j_{s+l+1}}^*.
\een
Clearly we can continue adding up the first two terms until we reach
\ben
\sigma_{j_{s+l}} (-q^{-1} \beta_{j_s}) + b_{s+l} (-q^{-1} \beta_{j_{s+l}})
+ b_{s+l+1} \beta_{j_{s+l+1}}^* = -q^{-1} \beta_{j_s} + b_{s+l+1}
\beta_{j_{s+l+1}}^*. 
\een
The $h_m$-pairing of the above expression with $\beta_{j_{s+l+1}}$
must be 0. We get $b_{s+l+1} = q^{-1} h_m(\beta_{j_s},
\beta_{j_{s+l+1}})$. This completes the proof.
\qed

\medskip
Let us denote by $W_\nu$ the matrix whose $(i,j)$-entry is 
\ben
W_{\nu,ij}:=
\begin{cases}
q^{-1}\, h_m (\beta_i(m),\beta_j(-m)) & \mbox{ if } u_i\in \Gamma_j(\eta_\nu),\\
1 & \mbox{ if } i=j,\\
0 & \mbox{ otherwise.} 
\end{cases}
\een
\begin{proposition}\label{prop:oscil-wallcrossing}
Suppose that $\eta$ and $\widetilde{\eta}$ are admissible directions separated by a single critical direction $\eta_\nu$. Then 
\ben
X_j(-\widetilde{\eta},t,z) = \sum_{i=1}^N X_i(-\eta,t,z) \, W_{\nu,ji}\quad 
\forall z\in H_{-\eta}\cap H_{-\widetilde{\eta}}.
\een
\end{proposition}
\proof
The 2nd formula in Proposition \ref{prop:beta-wallcrossing} implies that 
\beq\label{fund_p:wallcrossing}
I^{(m)}_{\widetilde{\beta}_j^*}(t,\lambda)= 
\sum_{k=1}^N I^{(m)}_{\beta_k^*}(t,\lambda) \, W_{\nu,jk},\quad 1\leq j\leq N. 
\eeq
Recalling Proposition \ref{prop:fund_periods}, c), we get that in
order to complete the proof it would be sufficient to deform the integration
contours $\gamma_j(-\widetilde{\eta})$ and $\gamma_k(-\eta)$ ($k:\
u_j\in \Gamma_k(\eta_\nu)$) to a common contour without changing the
values of the corresponding oscillatory integrals. This would be
possible thanks to our special choice of reference paths.

Let $u_i$ be the last entry of the $\eta_\nu$-sequence containing $u_j$. We pick a
contour $\gamma_i(-\widetilde{\eta},-\eta)$ consisting of 3 parts: the
ray $\Gamma_i(-\widetilde{\eta})\setminus{[u_i,u_i-\epsilon
  \widetilde{\eta})}$ with orientation from $\lambda=-\infty\,\widetilde{\eta} $ to
$\lambda=u_i-\epsilon \widetilde{\eta}$, the counter-clockwise arc
from $u_i-\epsilon \widetilde{\eta}$ to $u_i-\epsilon \eta$, and
finally the ray $\Gamma_i(-\eta)\setminus{[u_i,u_i-\epsilon \eta)}$ (see Figure \ref{fig:cont_def_wc}).
\begin{figure}[t]%% placement specifier
%% Use \includegraphics command to insert graphic files. Place graphics files in 
%% working directory.
\centering%% For centre alignment of image.
\includegraphics{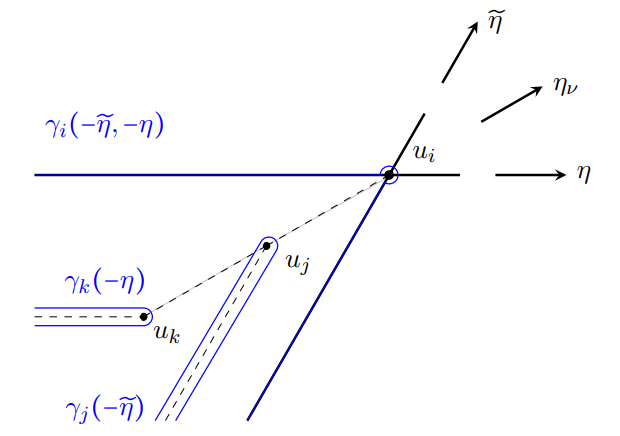}
%% Use \caption command for figure caption and label.
\caption{Contour deformation }\label{fig:cont_def_wc}
\end{figure}

Let $u_k$ be a point in the $\eta_\nu$-sequence preceding $u_j$. We claim that
\beq\label{deform_1}
\int_{\gamma_i(-\widetilde{\eta},-\eta )}
e^{\lambda/z}
I^{(m)}_{\beta_k^*}(t,\lambda) d\lambda =
\int_{\gamma_k (-\eta)}
e^{\lambda/z}
I^{(m)}_{\beta_k^*}(t,\lambda) d\lambda
\eeq
and
\beq\label{deform_2}
\int_{\gamma_i(-\widetilde{\eta},-\eta )}
e^{\lambda/z}
I^{(m)}_{\widetilde{\beta}_j^*}(t,\lambda) d\lambda =
\int_{\gamma_j(-\widetilde{\eta})}
e^{\lambda/z}
I^{(m)}_{\widetilde{\beta}_j^*}(t,\lambda) d\lambda.
\eeq
Let us justify the first identity. The argument for the second one is
similar. To begin with, note that $\beta_k^*$ is invariant under the monodromy
transformations $\sigma_l$ (the monodromy transformation corresponding
to the simple loop $C_l(\eta)$) for $l\neq k$. Thanks to our special
choice of the reference paths, i.e., the reference paths $C_l(\eta)$
$(l\neq k)$ do not
intersect the ray $\Gamma_k(-\eta)$, the fundamental group of
$\CC\setminus{\Gamma_k(-\eta)} $ is generated by the simple loops
corresponding to the paths $C_l(\eta)$ with $l\neq k$. Therefore, the period integral 
$I^{(m)}_{\beta_k^*}(t,\lambda) $ extends to a holomorphic function in
$\lambda$ for all $\lambda\in
\CC\setminus{\Gamma_k(-\eta)}$. In particular,
$I^{(m)}_{\beta_k^*}(t,\lambda) $ extends to a holomorphic function in
the domain $D$ bounded by the contours 
$\gamma_i(-\widetilde{\eta},-\eta )$ and $\gamma_k(-\eta)$. Furthermore, 
for $z\in H_{-\widetilde{\eta}}\cap H_{-\eta}$ the integrand $e^{\lambda/z}
I^{(m)}_{\beta_k^*}(t,\lambda)$ will have an exponential decay at
infinity in $D$. Therefore, the identity follows from the Cauchy
residue theorem. Finally, the formula that we have to prove follows
from Proposition \ref{prop:fund_periods}, c)  and formulas
\eqref{fund_p:wallcrossing}, \eqref{deform_1}, and
\eqref{deform_2}.
\qed

\medskip

Now we can express the Stokes matrix $V_+(\eta)$ in terms of the reflection
vectors. Let us first extend the definition of the critical directions
$\eta_\nu$ ($0\leq \nu\leq 2\mu-1$) by allowing arbitrary $\nu\in \ZZ$
so that $\eta_{\nu+2\mu} = \eta_\nu$. More precisely,
$\eta_{\nu+2\mu}$ is obtained from $\eta_\nu$ by clockwise rotation on
angle $2\pi.$ Such an extension is clearly
unique. Note that we have the following symmetry:
\beq\label{crit_directions:sym}
-\eta_\nu = \eta_{\nu-\mu}=\eta_{\nu+\mu},\quad \forall \nu\in \ZZ. 
\eeq
Recall that $X(\eta,t,z)$ is the matrix with columns $X_i(\eta,t,z)$. We
proved that $X(-\widetilde{\eta},t,z) = X(-\eta,t,z) \, W_\nu^t$ where
$\widetilde{\eta}$ is an admissible direction obtained from $\eta$ by
crossing the critical direction $\eta_\nu$ and $z\in H_{-\eta}\cap
H_{-\widetilde{\eta}}$. Note that $-\eta$ is 
obtained from $-\widetilde{\eta} $ by a clockwise
rotation. If $\widetilde{\eta}$ is rotated across $\eta_{\nu-1}$ to
$\widetilde{{\widetilde{\eta} }}$, then we get
\ben
X(- \widetilde{{\widetilde{\eta} }},t,z)=
X(-\widetilde{\eta},t,z) W_{\nu-1}^t= X(-\eta,t,z) \, W_\nu^t\,
W_{\nu-1}^t,\quad
z\in H_{-\eta}\cap H_{-\widetilde{{\widetilde{\eta} }}},
\een
where again $-\eta$ is obtained from $-\widetilde{{\widetilde{\eta} }}
$ via a clockwise rotation. Continuing in this way, i.e., rotating
$\eta$ anti-clockwise across all the
critical directions $\eta_\nu,\eta_{\nu-1},\dots, \eta_{\nu-(\mu-1)} =
-\eta_{\nu+1}$, we get 
\ben
X(\eta,t,z) = X(-\eta,t,z) W_\nu^t\, W_{\nu-1}^t\, \cdots \, W_{\nu-(\mu-1)}^t,
\een
where we may take $z\in H_{-\eta_\nu}\cap H_{\eta_{\nu+1}}$ because we
can start with $\eta$ sufficiently close to $\eta_\nu$ and at the end
cross $\eta_{\nu-(\mu-1)}=-\eta_{\nu+1}$ and stay sufficiently close
to $-\eta_{\nu+1}$. Recalling the definition of the Stokes matrix $V_-(\eta)$ we get
\beq\label{V-product}
V_-(\eta) = (W_\nu^t\, W_{\nu-1}^t\, \cdots \, W_{\nu-(\mu-1)}^t)^{-1},
\eeq
where $\eta$ is an admissible direction whose adjacent critical
directions are $\eta_\nu$ and $\eta_{\nu+1}$. Recalling the relation
$V_+=V_-^t$ (see Proposition \ref{prop:stokes_relation}) we get
\ben
V_+(\eta) = (W_{\nu-(\mu-1)} \, W_{\nu-(\mu-2)}\, \cdots\, W_\nu)^{-1}.
\een

\medskip

Slightly modifying the above argument we will obtain a simpler formula
for the Stokes matrices (see \cite{BJL1981}, Proposition 5). To begin
with, we need an analogue of Proposition \ref{prop:beta-wallcrossing}. If necessary
let us change the enumeration of the points
$u_1,\dots,u_N$  so that the following property holds: if we draw a line at $u_i$ parallel
to $\eta$ and we stand at $u_i$ looking towards infinity in the direction
$\eta$, then all points $u_j$ with $i<j$ (resp. $j<i$) will be in the
RHS (resp. LHS) half-plane. Note that $i<j$ is equivalent to
$\operatorname{Re}(u_i-u_j)/z<0$ for all $z\in H_{\eta_\nu}\cap
H_{-\eta_{\nu+1}}$. Therefore, recalling Proposition
\ref{prop:stokes_triangle} we get that $V_{+,ij}=0$ for $i>j$, that
is, $V_+$ is an upper-triangular matrix with ones on the diagonal. 

Let $\widetilde{\beta}_i(m)$ be the reflection vector corresponding to
the reference path obtained by composing $C_i(-\eta)$ and the
counter-clockwise oriented arc from
$\lambda^\circ(\eta)$ to $\lambda^\circ(-\eta)$ (see Figure
\ref{fig:rp}). Note that thanks to our choice of the 
indexes of $u_1,\dots,u_N$ we have
\beq\label{half-twist}
\widetilde{\beta}_1=\beta_1,\quad
\widetilde{\beta}_t = \sigma^{-1}_1\sigma^{-1}_2 \cdots \sigma^{-1}_{t-1}
(\beta_t)\quad 2\leq t\leq N.  
\eeq
Note that the formulas about $\widetilde{\beta}_j^*$ ($1\leq j\leq N$)
in Proposition
\ref{prop:beta-wallcrossing} were derived in a purely algebraic way
from the relations between $\widetilde{\beta}_j$ ($1\leq j\leq N$) and
$\beta_j$ ($1\leq j\leq N$). Therefore, in the current settings we
must have
\beq\label{dual-half-twist}
\widetilde{\beta}_N^* = \beta_N^*,\quad
\widetilde{\beta}_k^* =\beta_k^* +
\sum_{a=k+1}^N q^{-1} h_m(\beta_k(m), \beta_a(-m)) \beta_a^*. 
\eeq
Let us define the matrix $W_+$ of size $N\times N$ whose $(i,j)$ entry
is
\beq\label{Stokes_entries}
W_{+,ij} =
\begin{cases}
  q^{-1} h_m(\beta_i(m), \beta_j(-m)) & \mbox{ if }  i<j , \\
  1 & \mbox{ if } i=j,\\
  0 & \mbox{ otherwise }. 
\end{cases}
\eeq
Arguing in the same way as in the proof of Proposition
\ref{prop:oscil-wallcrossing} we get  $X(\eta,t,z)=X(-\eta,t,z) W_+^t$
for all $z\in H_{-\eta_\nu}\cap H_{\eta_{\nu+1}}$. Recalling the
definition of the Stokes matrix $V_-(\eta)$ we get that $V_-(\eta) =
(W_+^t)^{-1}$ and hence $V_+(\eta)=W_+^{-1}$, that is, formulas
\eqref{Stokes_entries} give the entries of the inverse Stokes matrix
$V_+(\eta)^{-1}$.

\medskip

Finally, we will finish this section by proving that the Stokes
matrices $V_+$ and $V_-$ are independent of $m$. More precisely, we
will express the pairings $h_m(\beta_k(m),\beta_j(-m))$ in terms of
the intersection pairing $(\ |\ )$. We follow the idea from the proof
of Lemma 2' in \cite{BJL1981}.
\begin{lemma}\label{le:q-dependence}
  Let $m\in \CC$ be a complex number with
  $\operatorname{Re}(m)<0$. Let us fix a negative integer $l$, such
  that, the real part of $\alpha=-m-1+l$ is positive. Then
  \ben
  I^{(m)}_{\beta_i(m)} (t,\lambda) = \int_{u_i}^\lambda
  \frac{(\lambda-s)^\alpha}{\Gamma(\alpha+1)}\,
  I^{(l)}_{\beta_i}(t,s) ds.
  \een
\end{lemma}
\proof
It is sufficient to prove the formula locally near $\lambda=u_i$. We
have a Laurent series expansion
\ben
I^{(l)}_{\beta_i}(t,s)=
\sqrt{2\pi} \sum_{k=0}^\infty \Psi R_k e_i\, 
\frac{(s-u_i)^{k-l-1/2}}{\Gamma(k-l+1/2)}.
\een
Using the substitution $x=\frac{\lambda-s}{\lambda-u_i}$ and the
standard formulas for the Euler $\beta$-integral we get
\ben
\int_{u_i}^\lambda
\frac{(\lambda-s)^\alpha}{\Gamma(\alpha+1)}
\frac{(s-u_i)^{k-l-1/2}}{\Gamma(k-l+1/2)}\, ds =
\int_0^1 x^\alpha (1-x)^{k-l-1/2} dx \,
\frac{ (\lambda-u_i)^{\alpha+k-l+1/2}  }{
  \Gamma(\alpha+1) \Gamma(k-l+1/2)}=
\frac{ (\lambda-u_i)^{\alpha+k-l+1/2}  }{
  \Gamma(\alpha+k-l+3/2)}.
\een
It remains only to note that $\alpha-l+1=-m$, that is, substituting
the Laurent series expansion of $I^{(l)}_{\beta_i}(t,s)$ and termwise
integrating in $s$ yields precisely the Laurent series of
$I^{(m)}_{\beta_i(m)}(t,\lambda)$.
\qed

\begin{proposition}\label{prop:hi_pairing}
The pairing $h_m$ takes the following form in the basis of reflection
vectors $\beta_i(m)$ ($1\leq i\leq N$) corresponding to the reference
paths $C_i(\eta)$  ($1\leq i\leq N$) :
\ben
h_m(\beta_k(m),\beta_j(-m)) =
\begin{cases}
  q (\beta_k|\beta_j) & \mbox{ if } k<j,\\
  q+q^{-1} & \mbox{ if } k=j,\\
  q^{-1} (\beta_k|\beta_j) & \mbox{ if } k>j.
\end{cases}
\een
\end{proposition}
\proof
Let $\lambda_0\in \Gamma_j(-\eta)$ be a point sufficiently close to
$u_j$. Let us choose $m$ such that its real part is sufficiently
negative. Let us consider the following difference
\beq\label{period_diff}
I^{(m)}_{\beta_k^+}(t,\lambda_0)-I^{(m)}_{\beta_k^-}(t,\lambda_0),
\eeq
where $\beta_k^+$ (resp. $\beta_k^-$)  means
that the value of the period is obtained from
$I^{(m)}_{\beta_k}(t,\lambda)$ via analytic 
continuation along a path which approaches $u_j$ along $C_j(\eta)$,
makes a small counter-clockwise (resp. clockwise) arc around $u_j$,
and continues towards $\lambda_0$ along $\Gamma_j(-\eta)$. We would
like to compute \eqref{period_diff} in two different ways. First, by
definition $\beta_k^+=\sigma_j(\beta_k^-) = \beta_k^--q^{-1}
h_m(\beta_k(m),\beta_j(-m))\beta_j^-$ where we put a sign $\beta_j^-$
to emphasize that the reference path should contain the clockwise arc
around $u_j$. We get that the difference \eqref{period_diff} coincides with
\beq\label{diff_1way}
-q^{-1} h_m(\beta_k(m),\beta_j(-m))\,
I^{(m)}_{\beta_j^-}(t,\lambda_0). 
\eeq
On the other hand, the analytic continuation can be computed using the
integral formula from Lemma \ref{le:q-dependence}. Namely,
\beq\label{ac}
I^{(m)}_{\beta_k^{\pm}} (t,\lambda_0) =
\int_{u_k}^{\lambda_0}
\frac{(\lambda_0-s)^\alpha}{\Gamma(\alpha+1)}\,
I^{(-l)}_k(t,s) ds,
\eeq
where the integration path is from $u_k$ to $\lambda^k(\eta)$ (see
Figure \ref{fig:rp}), the arc from $\lambda^k(\eta)$ to
$\lambda^j(\eta)$ (clockwise for $k<j$ and anti-clockwise for $k>j$),
the line segment approaching $u_j$ along $\Gamma_j(\eta)$, a small
clockwise (for $\beta_k^-$) or anti-clockwise (for  $\beta_k^+$) arc
around $u_j$, and finally a straight  line segment to
$\lambda_0$. Note that the integral splits into two
\ben
\int_{u_k}^{u_j}
\frac{(\lambda_0-s)^\alpha}{\Gamma(\alpha+1)}\,
I^{(-l)}_k(t,s) ds +
\int_{u_j}^{\lambda_0}
\frac{(\lambda_0-s)^\alpha}{\Gamma(\alpha+1)}\,
I^{(-l)}_{\beta_k^\pm}(t,s) ds,
\een
where the first integral does not depend on the choice of an arc
around $u_j$. Since $l$ is an integer, we have
$\beta_k^+(-l)-\beta_k^-(-l) = -(\beta_k|\beta_j) \beta_j^-(-l).$
Therefore, the difference \eqref{period_diff} takes the following
form:
\beq\label{diff_2way}
-(\beta_k|\beta_j)\, 
\int_{u_j}^{\lambda_0}
\frac{(\lambda_0-s)^\alpha}{\Gamma(\alpha+1)}\,
I^{(-l)}_{\beta_j^-}(t,s) ds.
\eeq
Note that $\operatorname{Arg}(\lambda_0-s)$ in the above formula is
obtained by continuously varying a small line segment $[s,\lambda_0]$
along the integration path in \eqref{ac}. The starting value of the argument is
$\operatorname{Arg}(\eta)$. If $k>j$, then the segment
$[s,\lambda_0]$ will be rotated anti-clockwise on angle $\pi$, so the
final value of $\operatorname{Arg}(\lambda_0-s)$ will be
$\operatorname{Arg}(\eta)+\pi$. If $k<j$, then the segment will be
rotated clockwise and the value of $\operatorname{Arg}(\lambda_0-s)$
will eventually become $\operatorname{Arg}(\eta)-\pi$.
Recalling Lemma \ref{le:q-dependence} we have
\ben
I^{(m)}_{\beta_j^-}(t,\lambda_0 ) =
\int_{u_j}^{\lambda_0}
\frac{(\lambda_0-s)^\alpha}{\Gamma(\alpha+1)}\,
I^{(-l)}_{\beta_j^-}(t,s) ds,
\een
where $\operatorname{Arg}(\lambda_0-s)$ should be
$\operatorname{Arg}(\eta)-\pi$.  The conclusion is that the
expression \eqref{diff_2way}, that is the difference \eqref{period_diff}, coincides with
$-e^{2\pi\ii \alpha} (\beta_k|\beta_j)\, I^{(m)}_{\beta_j^-}(t,\lambda_0 ) $
for $k>j$ and with
$- (\beta_k|\beta_j)\, I^{(m)}_{\beta_j^-}(t,\lambda_0 ) $
for $k<j$. Note that $e^{2\pi\ii \alpha}= e^{-2\pi\ii m} =
q^{-2}$. Comparing with our previous formula \eqref{diff_1way} we get
the statement of the proposition for the case when $k\neq j$. The case
$k=j$ was already considered (see Lemma \ref{le:local_Hm}).
\qed

\subsection{The central connection matrix}\label{sec:ccm}
Let us start by introducing Figure \ref{fig:cc_stokes} which might be
helpful in visualizing the constructions and following the arguments
in this section.  Let us identify $\CC=\RR^2$ in the standard way. There
are two kinds of objects on Figure \ref{fig:cc_stokes}: points and
vectors. We think of the points as elements of the $\lambda$-plane and
of vectors as elements of the $z$-plane. For example, a point in the sector $H_{\eta_\nu}\cap H_{-\eta_{\nu+1}}$ can be thought of as a vector in the 
shaded region on Figure \ref{fig:cc_stokes}.  

Suppose that $\eta$ is an admissible direction. Let us fix a value of $\log \eta$ such that \eqref{adm_Arg} holds and denote by $\beta_i(m)$ the twisted reflection vectors corresponding to the reference paths $C_i(\eta)$ (cf. Section \ref{sec:mdrv}, where the dependence on $\eta$ was explained very carefully). Let $X(\eta,t,z)$
be the matrix whose columns $X_i(\eta,t,z)$ are defined by the
oscillatory integrals  \eqref{osc_integral_i}. Note that the definition \eqref{osc_integral_i} is in fact independent of the choice of a value for $\log \eta$ as long as $\log (-z)$ and $\log \eta$ are defined by the same branch of the logarithm $\log:H_{-\eta}\to \CC$. We choose the branch such that \eqref{adm_Arg} holds. 
Using the clock-wise arc from $\eta$ to $-\eta$ we extend analytically 
$\log:H_{-\eta}\to \CC$ across the ray $-\ii \eta   \RR_{>0}$ (see Figure \ref{fig:cc_stokes}) to the entire half-plane $H_\eta$. Since $-z$ and $-\eta$ both belong to $H_\eta$ for $z\in H_{-\eta}$ we can use use the branch of $\log:H_\eta\to \CC$ introduced above to define $X(-\eta,t,z)$. We have
$X(\eta,t,z)\sim \Psi R e^{U/z}$ as $z\to 0$ and $z\in H_\eta$ and
$X(-\eta,t,z)\sim \Psi R e^{U/z}$ as $z\to 0$ and $z\in
H_{-\eta}$. 
\begin{figure}[t]%% placement specifier
%% Use \includegraphics command to insert graphic files. Place graphics files in 
%% working directory.
\centering%% For centre alignment of image.
\includegraphics{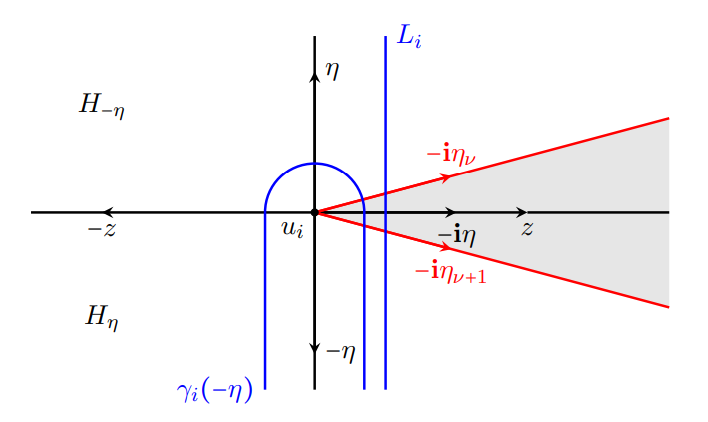}
%% Use \caption command for figure caption and label.
\caption{Points in the $\lambda$-plane and vectors in the $z$-plane}\label{fig:cc_stokes}  
\end{figure}
Recall that both $X(\eta,t,z)$ and $X(-\eta,t,z)$ extend analytically in $z\in 
H_{\eta_\nu}\cap H_{-\eta_{\nu+1}}$ by deforming the integration
contour of $X_i$ (see Section \ref{sec:AE}). We
get that $X(\eta,t,z)$, $X(-\eta,t,z)$, and $S(t,z) z^\theta z^{-\rho}$ are 
3 fundamental solutions to the Dubrovin connection analytic inside
the sector $H_{\eta_\nu}\cap H_{-\eta_{\nu+1}}$. Therefore, there
exist matrices $V_+(\eta)$ and $C(\eta)$ such that
\begin{align}
\label{con-m}
X(-\eta,t,z) & =  S(t,z) z^\theta z^{-\rho}\, C(\eta)^{-1},\\
\label{stokes-m}
X(-\eta,t,z) & = X(\eta,t,z) \, V_+(\eta) ,
\end{align}
for all $z\in H_{\eta_\nu}\cap H_{-\eta_{\nu+1}}$. The matrix
$V_+(\eta)$ is the Stokes matrix \eqref{V+} introduced earlier. 
Following Dubrovin
(see \cite{Du1999}) we will refer to $C(\eta)$ as the {\em central
  connection matrix}. The main goal in this section is to find a
formula for $C(\eta)$ in terms of the reflection vectors.

It is more convenient to work with the matrix $X^*(-\eta,t,z)$ whose
columns $X^*_i(-\eta,t,z)$ are defined in Proposition
\ref{prop:fund_periods}, c). According to Proposition
\ref{prop:fund_periods}
\ben
X^*(-\eta,t,z) = X(-\eta,t,z) = S(t,z)z^\theta
z^{-\rho} C(\eta)^{-1},
\een
where $z\in H_{\eta_\nu}\cap H_{-\eta_{\nu+1}}$. The key formula will
be proved in the following proposition (see \cite{Du1999}, Theorem
4.19). 
\begin{proposition}\label{prop:fund_cc}
Let $C^i(\eta)$ be the $i$th column of the matrix $C(\eta)^{-1}$. Then
\ben
\beta_i^*(m) =\sqrt{2\pi}\, (
q^{-1} e^{\pi\ii\theta} e^{\pi\ii\rho} +
q e^{-\pi\ii\theta} e^{-\pi\ii\rho} )^{-1} C^i(\eta).
\een
\end{proposition}
\proof
By definition (see Proposition \ref{prop:fund_periods})
\ben
\int_{\gamma_i(-\eta)}
e^{\lambda/z} I^{(m)}_{\beta_i^*} (t,\lambda) d\lambda =
-\sqrt{2\pi} q^{-1} (-z)^{-m+1/2}
X_i^*(-\eta,t,z),
\een
where $z\in H_{-\eta}$. Let us analytically extend the above identity
with respect to $z$ to the boundary of $H_{-\eta}$. To begin with, we choose $m$ such that $\operatorname{Re}(m)\gg 0$. Both sides of the equality that we have to prove are invariant under the shift $m\mapsto m+2$ so we can make $\operatorname{Re}(m)$ as big as we wish. For the analytic continuation of the RHS we use
that $X_i^*(-\eta,t,z)=X_i(-\eta,t,z)$. Suppose that $z\in
-\ii \eta \RR_{>0}$ is on the right (compared to the $\eta$-direction)
part of the boundary of $H_{-\eta}$. The analytic continuation of the
LHS is given by deforming the contour $\gamma_i(-\eta)$ to the line
$L_i:= u_i-\ii \eta \epsilon  + \eta \RR$ where $\epsilon>0$ is a
real number (see the blue contours on Figure \ref{fig:cc_stokes}). We get
\beq\label{L-integral}
\int_{L_i}
e^{\lambda/z} I^{(m)}_{\beta_i^*} (t,\lambda) d\lambda =
-\sqrt{2\pi} q^{-1} (-z)^{-m+1/2}
X_i(-\eta,t,z),
\eeq
where $z\in -\ii\eta\RR_{>0}. $ Here we used in an essential way that $\operatorname{Re}(m)\gg 0$. Indeed, for $\lambda\in L_i$ and $z\in -\ii \eta \RR_{>0}$ the exponential term $e^{\lambda/z}$ has constant length so in order to make sure that the integral on the LHS in \eqref{L-integral} is convergent we need to require that $I^{(m)}_{\beta_i^*(m)}(t,\lambda)$ is $L^1$-integrable on $L_i$. By definition the period vector has growth of type $O(\lambda^{\theta-m-1/2})$ so by arranging $\operatorname{Re}(m)$ to be sufficiently positive we can resolve the divergence problem. Moreover, the integrability will allow us to use the Fourier inversion formula. Let us recall that $X_i(-\eta,t,z) =
S(t,z) z^\theta z^{-\rho} C^i(\eta)$, where $\log z$ is determined
from the branch of $\log$ in $H_{-\eta}$ and the value $\log \eta$. 
Since we have restricted  $z\in -\ii\eta\RR_{>0}$ we get
$\operatorname{Arg}(z)=\operatorname{Arg}(\eta)-\tfrac{\pi}{2}$. On
the other hand, in formula \eqref{L-integral} the analytic branch of
$\log (-z)$ comes from the branch of $\log$ in $H_\eta$ induced from
$\log(-\eta) = \log \eta -\pi \ii$. In other words,
$\operatorname{Arg}(-z)=\operatorname{Arg}(z)-\pi$. Therefore, 
\ben
(-z)^{-m+1/2} = e^{-\pi \ii (-m+1/2)} \, z^{-m+1/2} = \ii^{-1} q
z^{-m+1/2}. 
\een
Let us rewrite \eqref{L-integral} as follows
\beq\label{FTI}
\frac{1}{2\pi\ii}
\int_{L_i}
e^{\lambda/z} I^{(m)}_{\beta_i^*} (t,\lambda) d\lambda =
\frac{1}{\sqrt{2\pi}} \,
S(t,z) z^{\theta-m+1/2} z^{-\rho} C^i(\eta).
\eeq
Let us substitute $w=1/z\in \tfrac{\ii}{\eta}\, \RR_{>0}$,
$\operatorname{Arg}(w) =
\tfrac{\pi}{2}-\operatorname{Arg}(\eta)$. We claim that 
$I^{(m)}_{\beta_i^*} (t,\lambda)$ is the
Laplace transform of the RHS, that is, 
\beq\label{LT-calibration}
I^{(m)}_{\beta_i^*} (t,\lambda)=
\frac{1}{\sqrt{2\pi}} \,
\int_0^\infty
e^{-x\lambda} S(t,x^{-1}) x^{m-\theta-1/2} x^\rho dx\,
C^i(\eta)\, ,
\eeq
where the integration is along the $w$-ray, that is, $\operatorname{Arg}(x) =
\tfrac{\pi}{2}-\operatorname{Arg}(\eta)$. Let us examine the
convergence of the above integral. Firstly, since the calibration
$S(t,x^{-1})$ is analytic at $x=0$, by choosing $m$ such that $\operatorname{Re}(m)\gg 0$, we can ensure that the integrand vanishes at $x=0$. When $x$ is close
to $\infty$, since the integrand is proportional to $X_i(-\eta,t,z)$,
it has at most exponential growth of order $e^{u_i x}$. Therefore, since $x\in \tfrac{\ii}{\eta}\RR_{\geq 0}$, the RHS of \eqref{LT-calibration} defines an analytic function, say $\widetilde{I}^{(m)}_{\beta_i^*}(t,\lambda)$,  for all
$\lambda$ in the half-plane $\operatorname{Re}((u_i-\lambda)\tfrac{\ii}{\eta})<0$. The function $\widetilde{I}^{(m)}_{\beta_i^*}(t,\lambda)$, being a Laplace transform of a smooth function vanishing at $x=0$ and having at most exponential growth at $x=\infty$, must be $L^1$-integrable along any line in the domain of convergence of the Laplace transform. Recalling the Laplace inversion formula, we get that formula \eqref{FTI} continues to hold if we replace $I^{(m)}_{\beta_i^*}(t,\lambda)$ by $\widetilde{I}^{(m)}_{\beta_i^*}(t,\lambda)$. Let us interpret the integral on the LHS of \eqref{FTI} as a Fourier transform. We get that $\widetilde{I}^{(m)}_{\beta_i^*}(t,\lambda)$ and $ I^{(m)}_{\beta_i^*}(t,\lambda)$, viewed as functions on $L_i\cong \RR$, are smooth $L^1$-integrable functions that have the same Fourier transform. Recalling the Fourier inversion formula, we get that $\widetilde{I}^{(m)}_{\beta_i^*}(t,\lambda) = I^{(m)}_{\beta_i^*}(t,\lambda).$ More precisely, if we let $f$ be the difference of the two functions, then $f$ is a smooth $L^1$-integrable function and its Fourier transform $\hat{f}=0$. The vanishing $\hat{f}=0$ means in particular that $\hat{f}$ is $L^1$-integrable. Since both $f$ and $\hat{f}$ are $L^1$-integrable, the Fourier inversion formula holds and $\hat{f}=0$ implies $f=0$. This completes the proof of our claim, i.e., formula \eqref{LT-calibration} holds for all $\lambda$ satisfying $\operatorname{Re}((u_i-\lambda)\tfrac{\ii}{\eta})<0$.

Recall that the calibration $S(t,x^{-1})$ is an entire function. 
Let us choose $\lambda\in \tfrac{\eta}{\ii}\RR_{> 0}$ with $|\lambda|\gg 0$ such that both $\operatorname{Re}((u_i-\lambda)\tfrac{\ii}{\eta})<0$ and $|\lambda|>|u_i|$ for all $1\leq i\leq N$. We claim that the RHS in \eqref{LT-calibration} can be computed by term-wise integrating the Taylor series expansion of $S(t,x^{-1})$ at $x=0$. To begin with, let us first expand $S(t,x^{-1})$ into a Taylor series at $x=0$ and compute explicitly the term-wise integrals. The justification of the term-wise integration will be done afterwards. Put $y= x \lambda$ and note that $y\in \RR_{>0}$. The RHS of \eqref{LT-calibration} takes the following form:
\beq\label{tw-integral}
\frac{1}{\sqrt{2\pi}} \,
\sum_{k=0}^\infty
S_k(t) \left(
  \int_0^\infty e^{-y} (y/\lambda) ^{k+m-\theta-1/2} (y/\lambda)^\rho
  \frac{dy}{\lambda}\,
\right)
C^i(\eta).
\eeq
Note that $(y/\lambda)^\rho= e^{\rho \log (y/\lambda)}$ and that
$\log(y/\lambda)$ can be produced by acting with the differential
operator $\partial_m:=\tfrac{\partial}{\partial m}$. The formula
transforms into
\ben
\frac{1}{\sqrt{2\pi}} \,
\sum_{k=0}^\infty
S_k(t) \left(
  \int_0^\infty e^{-y} (y/\lambda) ^{k+m-\theta-1/2} 
  \frac{dy}{\lambda}\,
\right)\cdot e^{\rho \overleftarrow{\partial}_m}
C^i(\eta),
\een
where the arrow over $\partial_m$ denotes {\em right} action of
the matrix differential operator $\rho \partial_m$. We have to
distinguish left and right action here because $\theta$ and $\rho$ do
not commute.  The above integral is just the definition of the
$\Gamma$-function. We get
\beq\label{twi-series}
\frac{1}{\sqrt{2\pi}} \,
\sum_{k=0}^\infty
S_k(t) \Big(
  \lambda^{\theta-m-k-1/2} \Gamma(m+k-\theta+1/2)
\Big)\cdot e^{\rho \overleftarrow{\partial}_m}
C^i(\eta).
\eeq
Formula \eqref{twi-series} will be simplified further but before doing this let us justify the term-wise integration. Since the LHS of \eqref{LT-calibration} is a solution to the second structure connection, the same is true for the series \eqref{twi-series}. In particular, the series must be absolutely convergent for all $|\lambda|>\operatorname{max}(|u_1|,\dots,|u_N|)$ because in the $\lambda$-direction the second structure connection is a Fuchsian connection with singularities at $\lambda=u_i$ ($1\leq i\leq N$) and $\lambda=\infty$. 
%This is the reason for requiring $\lambda$ to satisfy $|\lambda|>|u_i|$ for all $i$. The absolute convergence of \eqref{twi-series} is not enough to justify the term-wise integration! 
On the other hand, according to the Fubini theorem (see \cite{Ru1987}, Theorem 1.38 or Theorem 8.8), to justify the term-wise integration, it is sufficient to prove that if we replace all $S_k(t)$, $C^i(\eta)$, and the integrands in \eqref{tw-integral} by their absolute values, then the resulting series will be convergent. In other words, it is sufficient to prove that series of the following type are convergent:
\beq\label{FT-series}
\frac{1}{\sqrt{2\pi}}
\sum_{k=0}^\infty
|S_k(t)| \Big(|\lambda^{\theta-m-k-1/2}| \,
\Gamma(k+\alpha)\Big)\,  
|C^i(\eta)|\, ,
\eeq
where $\alpha$ is some constant depending on $\theta$, $\rho$, and $m$ but independent of $k$. Here by an absolute value of a matrix, we mean the maximal absolute value of its entries. Recall that the ratio of $\Gamma(k+\alpha)$ and $|\Gamma^{(l)}(k+\beta)|$ has at most polynomial growth in $k$ as $k\to \infty$ where $\Gamma^{(l)}(x)$ denotes the $l$th order derivative of $\Gamma(x)$ and $\beta$ is a fixed matrix. Since $\rho$ is nilpotent, only finitely many derivatives of the gamma function will appear in \eqref{twi-series}. We get that the absolute convergence of \eqref{twi-series} implies the convergence of \eqref{FT-series}. This completes the justification of the term-wise integration. 

Let us continue with the simplification of formula \eqref{twi-series}. Using the product formula
\ben
\Gamma(x) \Gamma(1-x)=
\frac{\pi}{\sin \pi x} =
\frac{2\pi\ii}{e^{\pi \ii x}-e^{-\pi\ii x} }
\een
with $x=\theta-m-k+1/2$ we get
\beq\label{LT-calibration_2}
\frac{1}{\sqrt{2\pi}} \,
\sum_{k=0}^\infty
S_k(t) (-1)^k\left(
  \frac{\lambda^{\theta-m-k-1/2} }{\Gamma(\theta-m-k+1/2)} \,
  \frac{2\pi}{e^{\pi\ii (\theta-m)}+ e^{-\pi\ii (\theta-m)}} 
\right)\cdot e^{\rho \overleftarrow{\partial}_m}\, 
C^i(\eta).
\eeq
Note that
\ben
\frac{2\pi}{e^{\pi\ii (\theta-m)}+ e^{-\pi\ii (\theta-m)}}  \,
e^{\rho \overleftarrow{\partial}_m} =
e^{-\rho \overleftarrow{\partial}_m}
\frac{2\pi}{
  e^{\pi\ii \theta}e^{-\pi\ii (m-\rho)} +
  e^{-\pi\ii \theta} e^{\pi\ii (m-\rho)}  } ,
\een
where we used that $e^{\pi\ii\theta} \rho = -\rho e^{\pi\ii\theta}$
thanks to the commutation relation $[\theta,\rho]=-\rho$. The left
action can be transformed into right action thanks to the following
formula:
\ben
\left(
  \frac{\lambda^{\theta-m-k-1/2} }{\Gamma(\theta-m-k+1/2)}
\right) \cdot e^{-\rho \overleftarrow{\partial}_m} =
e^{-\rho \partial_\lambda \partial_m}
\left(
  \frac{\lambda^{\theta-m-k-1/2} }{\Gamma(\theta-m-k+1/2)}
\right). 
\een
The above formula is proved by expanding $e^{-\rho\partial_m}=\sum_{l=0}^\infty
(-\rho)^l\partial_m^l/l!$, using the commutation relation $\theta
\rho^l = \rho^l(\theta-l)$, and finally noting that the shift
$k\mapsto k+1$ is equivalent to differentiation by $\lambda$. After
all these remarks, we can easily transform \eqref{LT-calibration_2}
into
\ben
\sqrt{2\pi} \,
\sum_{k=0}^\infty
S_k(t) (-1)^k  e^{-\rho \partial_\lambda \partial_m}
\left(
  \frac{\lambda^{\theta-m-k-1/2} }{\Gamma(\theta-m-k+1/2)} \,
\right)\cdot\,
(q^{-1} e^{\pi\ii \theta}e^{\pi\ii \rho} +
 q e^{-\pi\ii \theta} e^{-\pi\ii \rho}  )^{-1} 
C^i(\eta).
\een
The infinite sum over $k$ is precisely our definition of the
fundamental solution $I^{(m)}(t,\lambda)$ of the second structure
connection $\nabla^{(m)}$. The formula for $\beta_i^*(m)$ is essentially proved except for the following subtle point: we have to check that the value of $\log \lambda$ in the above formula agrees with the value of $\log \lambda$ specified by the reference path $C_i(\eta)$. Going back to our choice of $\lambda$, we get that 
$\log \lambda =\ln |\lambda| + \ii (\operatorname{Arg}(\eta)-\tfrac{\pi}{2})$. On the other hand, the ray $\tfrac{\eta}{\ii}\RR_{>0}$ along which we allowed $\lambda$ to vary contains the point $\lambda^\circ(\eta)=\tfrac{\eta}{\ii} \lambda^\circ\in C_i(\eta)$. By definition, $\lambda^\circ(\eta)$ is connected to $\lambda^\circ$ by continuously varying $\eta$ to $\ii$ so that the argument of $\eta$ varies in $(-\tfrac{\pi}{2},\tfrac{3\pi}{2}]$(cf. Section \ref{sec:mdrv}). Therefore, $\log \lambda^\circ(\eta) =\ln \lambda^\circ + \ii (\operatorname{Arg}(\eta)-\tfrac{\pi}{2})$, that is, the branch of $\log \lambda$ along the ray 
$\tfrac{\eta}{\ii}\RR_{>0}$ induced by the reference path $C_i(\eta)$ agrees with the branch coming from the term-wise integration of the RHS of formula \eqref{LT-calibration}. This completes the proof. 
\qed

\medskip
Now we are in position to derive the precise formulas relating the monodromy data of the 1st and the 2nd structure connections. Let us first state the following simple but very useful formula:
\beq\label{conj_Tt}
g\, A^T = A^t\, g,\quad A\in \operatorname{End}(H)\cong\operatorname{Mat}_{N\times N}(\CC),
\eeq
where $g$ is the matrix of the Frobenius pairing, that is, $g_{ij}=(\partial/\partial t_i,\partial/\partial t_j)$, ${}^T$ is transposition with respect to the Frobenius pairing, and ${}^t$ is the usual transposition of matrices. In the above identity, we use a fixed basis of
flat vector fields $\phi_i:=\partial/\partial t_i$ $(1\leq i\leq N)$ to identify the space $\operatorname{End}(H)$ of linear operators in $H$ with the space $\operatorname{Mat}_{N\times N}(\CC)$ of matrices of size $N\times N$. 
Let us also introduce the following convenient notation:
\ben
A^{-T}=(A^{-1})^T,\quad 
A^{-t}=(A^{-1})^t,\quad 
A\in \operatorname{Mat}_{N\times N}(\CC).
\een
Let us recall the well known relations between the Stokes matrices and
the central connection matrix (see \cite{Du1999}). 
\begin{proposition}\label{prop:stokes_cc}
The following formulas hold:
\ben
V_+ & = & C^{-t} g 
e^{-\pi\ii\rho} e^{-\pi \ii \theta} C^{-1},\\
V_- & = & C^{-t} g 
e^{\pi\ii\rho} e^{\pi \ii \theta} C^{-1}.
\een
\end{proposition}
\proof 
The 2nd formula follows from the first one and the relation $V_-=V_+^T$. Let us prove the first formula. 
Suppose that $z\in H_{\eta_\nu}\cap H_{-\eta_{\nu+1}}$. We have
\beq\label{sympl_V+}
X(-\eta,t,-z)^t g X(-\eta,t,z) = 
(X(\eta,t,-z) V_-)^t
g X(-\eta,t,z)= V_-^t = V_+,
\eeq
where we used the quadratic relation $X(-\eta,t,-z)^t g X(-\eta,t,z)=1$ and the relation $V_-^t=V_+$ -- see Proposition \ref{prop:stokes_relation} and its proof. On the other hand, by definition, we have $X(-\eta,t,z)=S(t,z) z^\theta z^{-\rho} C^{-1}$. Let us analytically extend this identity in $z$ from $z$ to $-z$ along the anti-clockwise arc. We get a second identity of the form
$X(-\eta,t,-z)= S(t,-z) z^{\theta} e^{\pi\ii\theta} z^{-\rho} e^{-\pi\ii\rho}C^{-1}$. Substituting these two formulas in \eqref{sympl_V+} we get 
\ben
V_+= (S(t,-z) z^{\theta} e^{\pi\ii\theta} z^{-\rho} e^{-\pi\ii\rho}C^{-1})^t 
g
S(t,z) z^\theta z^{-\rho} C^{-1}= C^{-t} g 
e^{-\pi\ii\rho} e^{-\pi \ii \theta} C^{-1},
\een  
where we used repeatedly formula \eqref{conj_Tt}, the symplectic condition $S(t,-z)^T S(t,z)=1$, and the relation $e^{\pi\ii\theta} \rho =-\rho e^{\pi\ii\theta}$. 
\qed

\medskip
Let us introduce the matrix $h(m)$ of the pairing $h_m$, that is, 
$h_{ij}(m):=h_m(\phi_i,\phi_j)$. Let $\beta(m)=[\beta_1(m),\dots,\beta_N(m)]$ be the matrix with columns the reflection vectors $\beta_i(m)$, that is, the entries of $\beta(m)$ are defined by $\beta_i(m)=:\sum_{k=1}^N \beta_{ki}(m) \phi_k$. Similarly, let $\beta^*=[\beta_1^*(m),\dots,\beta_N^*(m)]$ be the matrix whose columns are given by the dual vectors $\beta_i^*(m)$.
Using formula \eqref{Stokes_entries} for the entries of $W_+=V_+^{-1}$ and Propositions \ref{prop:hi_pairing}, we get that
$h_m(\beta_i(m) ,\beta_j(-m))$ coincides with the $(i,j)$-entry of  $q V_+^{-1} + q^{-1}V_+^{-t}$.  Therefore,  
\beq\label{hm_Stokes}
q V_+^{-1} + q^{-1}V_+^{-t}= \beta(m)^t h(m) \beta(-m)
\eeq
On the other hand, since by definition $\beta^*(m)^t h(m) \beta(-m)=1$, we get 
$\beta(m)^t=(q V_+^{-1} + q^{-1}V_+^{-t})\beta^*(m)^t$, that is, 
\ben
\beta(m)= \beta^*(m) (q V_+^{-t} + q^{-1}V_+^{-1}). 
\een
Recalling Proposition \ref{prop:stokes_cc} we get 
\ben
q V_+^{-t} + q^{-1}V_+^{-1}= 
C(
q \, e^{-\pi\ii \theta} e^{-\pi\ii\rho} + 
q^{-1} \, e^{\pi\ii \theta} e^{\pi\ii\rho})g^{-1} C^t. 
\een
Recalling Proposition \ref{prop:fund_cc} we get
\beq\label{refl_cc}
\beta(m)=\sqrt{2\pi}  (
q \, e^{-\pi\ii \theta} e^{-\pi\ii\rho} + 
q^{-1} \, e^{\pi\ii \theta} e^{\pi\ii\rho})^{-1} C^{-1}
(q V_+^{-t} + q^{-1}V_+^{-1})=
\sqrt{2\pi}g^{-1} C^t.
\eeq

Now we are in position to state and prove the main result of our paper. Let us quickly recall our assumptions. We have a semi-simple Frobenius manifold such that its grading operator is a Hodge grading operator (Cf. Definition \ref{def:hgo}). We fix a reference point $t^\circ\in M$ such that $\eta^\circ:=\ii$ is an admissible direction. Let $\eta$ be an admissible direction and $\beta_1(m),\dots,\beta_N(m)$ be a set of $m$-twisted reflection vectors corresponding to the reference paths $C_1(\eta),\dots,C_N(\eta)$ (Cf. Section \ref{sec:mdrv}). Finally, the admissible direction determines the central connection matrix $C=C(\eta)$ and the Stokes matrix $V_+=V_+(\eta)$ defined by \eqref{con-m}--\eqref{stokes-m}.  
\begin{theorem}\label{thm:refl_cc}
a)
The $(i,j)$-entry of the central connection matrix is related to the components of the reflection vectors by the following formula:
\ben
C_{ij}=\frac{1}{\sqrt{2\pi}} \, (\beta_i(m),\phi_j). 
\een

b) The pairing $h_m$ can be computed by the following formula:
\ben
h_m(a,b)= 
q \langle a, b\rangle + 
q^{-1} \langle b, a\rangle,\quad \forall a,b\in H,
\een
where $\langle\ ,\ \rangle$ is the Euler pairing \eqref{euler_p}.

c) The reflection vectors $\beta_i(m)$ ($1\leq i\leq N$) are
independent of $m$ and the Gram matrix of the Euler pairing is upper-triangular:
\ben
\langle \beta_i, \beta_j\rangle = 0 \quad \forall i>j,
\een 
with $1$'s on the diagonal: $\langle \beta_i, \beta_i\rangle = 1$.

d) The Gram matrix of the Euler pairing in the basis $\beta_i$ ($1\leq
i\leq N$) coincides with the inverse Stokes matrix $V_+^{-1}$. 
\end{theorem}
\proof
a) According to \eqref{refl_cc} we have $C=\tfrac{1}{\sqrt{2\pi}} \beta(m)^t g$. Comparing the $(i,j)$ entries in this matrix identity we get the formula stated in part a). 

b) According to \eqref{hm_Stokes} we have 
\ben
h(m)=\beta(m)^{-t}(q V_+^{-1} + q^{-1}V_+^{-t}) \beta(-m)^{-1}.
\een
Recalling Proposition \ref{prop:stokes_cc} we get 
\ben
q V_+^{-1} + q^{-1}V_+^{-t}= 
C(
q \, e^{\pi\ii \theta} e^{\pi\ii\rho} + 
q^{-1} \, e^{-\pi\ii \theta} e^{-\pi\ii\rho})g^{-1} C^t. 
\een
Finally, since $\beta(m)= \sqrt{2\pi} g^{-1} C^t$ we get 
\ben
h(m)=\frac{1}{2\pi}
g C^{-1} \, C(  
q \, e^{\pi\ii \theta} e^{\pi\ii\rho} + 
q^{-1} \, e^{-\pi\ii \theta} e^{-\pi\ii\rho})
g^{-1} C^t\, 
C^{-t} g=
\frac{1}{2\pi}
g (  
q \, e^{\pi\ii \theta} e^{\pi\ii\rho} + 
q^{-1} \, e^{-\pi\ii \theta} e^{-\pi\ii\rho}).
\een 
The above formula implies that 
\ben
h_m(\phi_i,\phi_j) =\frac{1}{2\pi} (\phi_i, 
(  
q \, e^{\pi\ii \theta} e^{\pi\ii\rho} + 
q^{-1} \, e^{-\pi\ii \theta} e^{-\pi\ii\rho})\phi_j)=
q\langle \phi_i, \phi_j\rangle + 
q^{-1} \langle \phi_j, \phi_i\rangle. 
\een
c) 
The fact that $\beta_i(m)$ is independent of $m$ follows immediately from part a) because the central connection matrix is independent of $m$. The rest of the statement is an immediate consequence of Proposition \ref{prop:hi_pairing} and part b). Indeed, if $i<j$, then we have 
\ben
q(\beta_i|\beta_j)= 
q\langle \beta_i,\beta_j\rangle + 
q^{-1} \langle \beta_j,\beta_i\rangle.
\een
On the other hand, recalling part b) with $q=1$ we get $(\beta_i|\beta_j)= 
\langle \beta_i,\beta_j\rangle + 
\langle \beta_j,\beta_i\rangle$. The above identity is possible if and only if $\langle \beta_j,\beta_i\rangle=0$. Similarly, if $i=j$, then we have 
\ben
q+q^{-1}= q\langle \beta_i,\beta_i\rangle + 
q^{-1} \langle \beta_i,\beta_i\rangle= (q+q^{-1})\langle \beta_i,\beta_i\rangle
\een
which implies that $\langle \beta_i,\beta_i\rangle=1$.

d) This part is an immediate consequence of formula
\eqref{Stokes_entries} and parts b) and c). 
\qed

\begin{remark}\label{rem:Du1999}
Let us compare our notation to Dubrovin's one in \cite{Du1999}. If
$\eta$ is an admissible direction, then $\ell_+= \ii
\eta^{-1}\RR_{>0}$ is the positive part of an admissible line in the
sense of Dubrovin. Then $X(\eta,t,z)= Y_{\rm left} (t,z^{-1})$,
  $X(-\eta,t,z)= Y_{\rm right} (t,z^{-1})$, and $S(t,z)z^\theta z^{-\rho}
=Y_0(t,z^{-1})$. Note that the degree operator in Dubrovin is
$\mu=-\theta$ while the nilpotent operators coincide $R=\rho$. It
follows that the inverse Stokes matrix $V^{-1}_+$ coincides with Dubrovin's
Stokes matrix $S$. Finally, the central connection matrix in our notation coincides
with the Dubrovin's one.  \qed
\end{remark}

\section{Dubrovin conjecture}

Following \cite{Dubrovin1998} (see also \cite{CDG2024}) we present the so-called Dubrovin conjecture. Roughly speaking, Dubrovin conjecture relates the big quantum cohomology of a variety $X$, as a Frobenius manifold, with its bounded derived category of coherent sheaves. The main goal is to give a reformulation of the conjecture in terms of the language introduced in the previous section. 

\subsection{Quantum cohomology}\label{sec:q_coh}

We recall some basic aspects about quantum cohomology, for a more detail account, for instance see \cite{Cox-Katz}. Let $X$ be a smooth projective algebraic variety, with vanishing odd cohomology, and $\overline{\mathcal{M}}_{g,n}(X, d)$ be the Deligne-Mumford stack of $n$-pointed stable maps of genus $g$ representing a class $d \in H_2(X, \mathbb{Z})$. Let us consider the evaluation maps 
$\op{ev}_i: \overline{\mathcal{M}}_{g,n}(X, d) \rightarrow X$, $i= 1, \dots, n$, and the map $\op{ev} := \op{ev}_1 \times \cdots \times \op{ev}_n: \overline{\mathcal{M}}_{g,n}(X, d) \rightarrow X^n$. Then, the descendant Gromov-Witten invariant is defined by the following formula: 
\ben
\langle \alpha_1 \psi^{l_1}, \dots, \alpha_n \psi^{l_n}\rangle_{g,n,d} = 
\int_{[ \overline{\mathcal{M}}_{g,n}(X, d)]^{\mathrm{vir}}} 
\psi_1^{l_1}\cdots \psi_n^{l_n}\, 
\op{ev}^{*}(\alpha_1 \times \cdots \times \alpha_n),
\een
where $[\overline{\mathcal{M}}_{g,n}(X, d)]^{\mathrm{vir}}$ is the virtual fundamental class in the Chow ring $CH_{*}( \overline{\mathcal{M}}_{g,n}(X, d))$ constructed in \cite{BF_intnormalcone} and $\psi_i$ is the first Chern class of the tautological line bundle formed by the cotangent line at the $i$-th marked point. In particular, genus-0 Gromov-Witten invariants (with no descendants) can be used to define a deformation of the cup product in cohomology.
\begin{definition} Let $\omega$ be a complexified K\"{a}hler class on a smooth projective variety $X$. Let $\phi_1 = 1, \dots, \phi_N$ be a basis of $H^{*}(X, \mathbb{C})$ and $\tau = \sum_{i=1}^N t_i \phi_i$. Then, the Gromov-Witten potential $\Phi$ is defined by the following formula:
\ben
\Phi(\tau) = \sum_{n=0}^{\infty} \sum_{d \in H_2(X, \mathbb{Z})} \frac{1}{n!} \langle \tau^n\rangle_{0,n,d} q^{d}
\een
where $\langle \tau^n\rangle_{0,n,d} = \langle \tau, \dots, \tau \rangle_{0,n,d}$ (with $\tau$ taken $n$ times) and $q^{d} = e^{ 2 \pi \ii \int_{d} \omega}$. \qed
\end{definition}
%For reasons that will become clearer later, from now on, we will assume
If $X$ is a Fano variety, then there are only a finite number of $d$'s such that $\langle \tau^n\rangle_{0,n,d} \neq 0$, so $\Phi \in \mathbb{C}[\![t_1, \dots, t_N ]\!]$, where the $t_i$'s are the formal variables associated to the basis $\phi_i$ $(1\leq i\leq N)$. Therefore $\Phi$ can be considered as a function on a formal neighbourhood of $0 \in H^*(X, \mathbb{C})$. In general, we fix an ample basis $p_1,\dots,p_r$ of $H^2(X,\ZZ)\cap H^{1,1}(X,\CC)$ and introduce the so called Novikov variables $q_1,\dots,q_r$. The expression 
$q^d:=
q_1^{\langle p_1,d\rangle}\cdots 
q_r^{\langle p_r,d\rangle}$ is interpreted as an element in the ring of formal power series $\CC[\![q]\!]:=\CC[\![q_1,\dots,q_r]\!]$ and the potential $\Phi$ is considered as a formal power series in the ring
\ben
\CC[\![q,t]\!]:=\CC[\![t_1, q_1 e^{t_2},\dots, q_r e^{t_{r+1}},t_{r+2},\dots,t_N]\!],
\een
where we identified $p_i=\phi_{i+1}$ ($1\leq i\leq r$) and we used the divisor equation to express $\Phi$ as a function of $q_i e^{t_{i+1}}$. It is believed that the Gromov--Witten potential is convergent (see below for a more precise statement). In case of convergence, the complexified K\"ahler class is related to the Novikov variables via 
$\omega=\tfrac{1}{2\pi\ii}\left( 
p_1\log q_1+\cdots +p_r\log q_r\right)$.  

\begin{definition} The big quantum cohomology of $X$ is the ring $H^*(X, \mathbb{C}[\![q,t]\!])$, 
with the product given on generators by $\phi_i \bullet \phi_j = \sum \frac{\partial^3 \Phi}{\partial{t_i}\partial{t_j}\partial{t_k}} \phi^k $, where $\phi^1, \dots, \phi^N$ form a Poincaré dual basis to $\phi_1, \dots, \phi_N$. We will denote this ring by $QH^*(X)$.\qed
\end{definition}

\begin{remark}
If we set $\delta = \sum_{i=1}^r t_{i+1} p_i$ and $\epsilon = t_1 \phi_1 + \sum_{i=r+1}^Nt_i\phi_i$, then 
\ben
\phi_i \bullet \phi_j = \sum_k \sum_{n=0}^{\infty} \sum_{d} \frac{1}{n!} \langle \phi_i,\phi_j,\phi_k,\epsilon^n \rangle_{0,n+3,d} e^{\int_{d}\delta}q^{d}\phi^k.
\een
In addition, if we set $\epsilon = 0$, then we get that
\ben
\phi_i \bullet \phi_j\vert_{\epsilon=0} = 
\sum_k \sum_{d} \langle \phi_i,\phi_j,\phi_k \rangle_{0,n+3,d} e^{\int_{d}\delta}q^{d}\phi^k.
\een
This restriction is known as the small quantum product and the corresponding ring is called small quantum cohomology ring. In fact, since the big quantum product is defined formally on $H^*(X,\mathbb{C}[\![q,t]\!])$ in terms of $t_1, \dots, t_N$ and $q^{d}$, the small quantum cup product is obtained by restricting to $H^2(X,\mathbb{C}[\![q]\!])$, that is, setting $t_1=t_{r+2} = \cdots = t_N = 0$ in the formula for $\phi_i \bullet \phi_j$. This is equivalent to setting $\epsilon = 0$.\qed
\end{remark}

Suppose now that the Novikov variables $q_1=\cdots=q_r=1$ and that there exists a non-empty open subset $M \subseteq H^*(X,\mathbb{C})$ where the Gromow-Witten potential $\Phi$ converges. More precisely, we assume that $M$ contains $\tau\in H^*(X,\CC)$ such that $e^{t_{i+1}}$ ($1\leq i\leq r$) and $t_j$ ($r+2\leq j\leq N$) are complex numbers with sufficiently small length ($t_1$ could be arbitrary because $\Phi$ is polynomial in $t_1$). Let 
\ben
g : H^*(X, \mathbb{C}) \times H^*(X, \mathbb{C}) \rightarrow \mathbb{C},
\quad 
g (\xi,\zeta) = \int_X \xi \cup \zeta
\een 
be the Poincaré pairing which will be taken as a Frobenius pairing. Put
\ben
E = c_1(X) + \sum_{i=1}^N \Big(
1 - \mathrm{deg}_\CC (\phi_i) \Big)
t_i \frac{\partial}{\partial t_i},
\een
where $\op{deg}_\CC$ is half of the standard cohomology degree. 
Then, $M$ equipped with the big quantum cup product defined above, the Poincaré pairing, and the Euler vector field $E$ is a Frobenius manifold of conformal dimension $D:=\op{dim}_\CC(X)$. The semi-simplicity of the quantum cup product is an indication that the target smooth algebraic variety $X$ has many rational curves. Therefore, from the point of view of birational geometry, it is very important to understand when is the Frobenius manifold underlying the big quantum cohomology semisimple? 
Before trying to approach an answer to this question, we need to recall some background on the bounded derived category of coherent sheaves of $X$. 

\subsection{Derived categories}

For more details about derived categories we refer to \cite{GM}. Let $\mathcal{T}$ be a $\mathbb{C}$-linear triangulated category. Let us recall the following notation. Given an object $E\in \mathcal{T}$ put $E[k]:=T^kE$ where $T$ is the translation functor of the triangulated category. Furthermore, $\op{Hom}(E,F)$ denotes the complex vector space of morphisms in $\mathcal{T}$ from $E$ to $F$. Let us introduce also the complex $\op{Hom}^\bullet(E,F)$ of vector spaces with a trivial differential whose component in degree $k$ is $\op{Hom}^k(E,F):=\op{Hom}(E,F[k])$.
\begin{definition}
  An object $E$ in $\mathcal{T}$ is called {\em exceptional} if it satisfies the following conditions:
  \ben
  \op{Hom}^k(E,E)=0 \mbox{ for } k\neq 0,\quad
  \op{Hom}(E,E)=\CC.\qed
  \een
\end{definition}
\begin{definition}
  A sequence of objects $(E_1, \dots, E_N)$ is called an {\em exceptional collection} if every object $E_i$ is exceptional and $\mathrm{Hom}^{\bullet}(E_i,E_j) = 0$ for $i >j$.
An exceptional collection is said to be {\em full} if it generates $\mathcal{T}$ as a triangulated category.  \qed
\end{definition}
Following Bondal (see \cite{Bo1990}) we would like to recall the mutation operations. An exceptional collection $(E,F)$ consisting of two objects is said to be an {\em exceptional pair}. Let $(E,F)$ be an exceptional pair. We define objects $L_EF$ and $R_FE$ such that the following sequences are distinguished triangles
\ben
\xymatrix{
L_EF\ar[r]  &
\op{Hom}^\bullet (E,F)\otimes E \ar[r] &
F ,\\
E\ar[r] &
\op{Hom}^\bullet (E,F)^*\otimes F \ar[r] & R_FE,}
\een
where for a complex of vector spaces $V^\bullet$ we denote by $V^k\otimes E[-k]$ the direct sum of $\op{dim}(V^k) $ copies of $E[-k]$ and by $V^\bullet \otimes E$ the direct sum of all $V^k\otimes E[-k]$. The map $ \op{Hom}^\bullet (E,F)\otimes E \to F$ is induced from the tautological maps $\op{Hom}(E,F[k])\otimes E[-k] \to F$, that is, fix a basis $f_i$ of $ \op{Hom} (E,F[k])=\op{Hom}(E[-k],F)$, then $\oplus_i f_i$ is a morphism $\oplus_i E[-k] \to F$.  Similarly, the map
$E\to \op{Hom}^\bullet (E,F)^*\otimes F$ is induced from the tautological maps
$E\to \op{Hom}(E,F[-k])^* \otimes F[-k]$, that is, fix a basis $f^i$ of $\op{Hom}(E,F[-k])^*$ and a dual basis $f_i$ of $\op{Hom}(E,F[-k])$, then $\oplus_i f_i$ is a morphism from $E\to \oplus_i F[-k]$. Note that taking the dual changes the sign of the grading: $\op{Hom}(E,F[k])^*$ is in degree $-k$ while $\op{Hom}(E,F[k])$ is in degree $k$.
It is easy to check that both $(L_E F,E)$ and $(F,R_F(E))$ are exceptional pairs. More generally, given an exceptional collection $\sigma=(E_1,\dots,E_N)$ we define left mutation $L_i$ and right mutation $R_i$ by mutating the adjacent objects $E_i$ and $E_{i+1}$, that is, 
\ben
L_i\sigma & = &  (E_1,\dots, E_{i-1}, L_{E_i} E_{i+1}, E_i, E_{i+2},\dots,E_N),\quad
1\leq i\leq N-1, \\
R_i\sigma & = &  (E_1,\dots, E_{i-1}, E_{i+1}, R_{E_{i+1}} E_i, E_{i+2},\dots,E_N),\quad
1\leq i\leq N-1.
\een
It turns out that these operations define an action of the braid group of $N$ strings on the set of exceptional collections, that is, the following commutation relations hold (see \cite{Bo1990}, Assertion 2.3):
\ben
R_i L_i=1\quad (1\leq i\leq N-1)
\een
and
\ben
R_i R_{i+1} R_i = R_{i+1} R_i R_{i+1}, \quad
L_i L_{i+1} L_i = L_{i+1} L_i L_{i+1},
\een
where $1\leq i\leq N-2$. Using mutations we can construct the so-called {\em left Koszul dual} of the exceptional sequence $\sigma=(E_1,\dots,E_N)$, that is, the exceptional sequence defined by 
\ben
\widetilde{\sigma}:= L_{N-1} (L_{N-2}L_{N-1})\cdots (L_1 L_2\cdots L_{N-1}) (\sigma)
\een
is called the left Koszul dual of $\sigma$. More explicitly,
\ben
\widetilde{\sigma}=(
\widetilde{E}_N,\dots, \widetilde{E}_1),\quad
\widetilde{E}_1:=E_1,\quad
\widetilde{E}_i:= L_{E_1} \cdots L_{E_{i-1}} (E_i)\quad (2\leq i\leq N).
\een
In other words, using left translations (see \cite{Bo1990}, Section 2), we are moving first $E_N$ to the left through $E_1,\dots,E_{N-1}$, then in the resulting collection we move $E_{N-1}$ to the left through $E_1,\dots,E_{N-2}$ etc.. We are not going to use it but let us point out that one can define a right Koszul dual in a similar way, that is,  $R_1(R_2R_1)\cdots (R_{N-1} R_{N-2}\cdots R_1) (\sigma)$. As we will see now, the left (resp. right) Koszul dual corresponds to changing the admissible direction $\eta$ to $-\eta$ by an anti-clockwise (resp. clockwise) rotation.  
\begin{definition}
Let $(E_1, \dots, E_N)$ be a full exceptional collection. The \emph{helix} generated by $(E_1, \dots, E_N)$ is the infinite collection $(E_i)_{i \in \mathbb{Z}}$ defined by the iterated mutations
\ben
E_{i+N}= R_{E_{i+N-1}} \cdots R_{E_{i+1}}E_i, \\
E_{i-N}= L_{E_{i-N+1}} \cdots L_{E_{i-1}}E_i
\een
A \emph{foundation} of a helix is any family of $N$ consecutive objects $(E_{i+1}, E_{i+2}, \dots, E_{i+N})$. The collection $(E_1, \dots, E_N)$ is called the \emph{marked foundation}. \qed
\end{definition}

\begin{definition}
Let $[\mathcal{T}]$ be the set of isomorphism classes of objects of $\mathcal{T}$. 
The Grothendieck group $K_0(\mathcal{T})$ of $\mathcal{T}$ is defined as the quotient of the free abelian group generated by $[\mathcal{T}]$ and the Euler relations: $[B] = [A] + [C]$ whenever there exist a triangle $A \rightarrow B \rightarrow C \rightarrow A[1]$ in $\mathcal{T}.$\qed
\end{definition}
We can define the so-called Grothendieck-Euler-Poincaré pairing as
\ben
\chi(E,F) = \sum_i (-1)^i \mathrm{dim}_{\mathbb{C}} \mathrm{Hom}^i(E,F)
\een
for any pair of objects $E$ and $F$ in $\mathcal{T}$. Note that on the level of $K$-theoretic groups the left and right mutations take the form
\ben
[L_EF] = [F]-\chi(E,F) [E],\quad
[R_F E]=[E]-\chi(E,F) [F].
\een
\begin{lemma}\label{le:left_KD}
  Let $\sigma=(E_1,\dots,E_N)$ be a full exceptional collection and $\widetilde{\sigma}=(\widetilde{E}_N,\dots,\widetilde{E}_1)$ be its left Koszul dual. Then
  $\chi(E_i, \widetilde{E}_j)=\delta_{ij}$. 
\end{lemma}
\proof
Let
$B_i=[E_i]\in K_0(\mathcal{T})$ and
$\widetilde{B}_i=[\widetilde{E}_i]\in K_0(\mathcal{T})$. Let us introduce the reflection $\sigma_i(A):=A-q^{-1}h(A,B_i)B_i$ where $h(A,B):=q \chi(A,B)+q^{-1}\chi(B,A)$ and $q$ is generic such that $h(\ ,\ )$ is a non-degenerate pairing. Using that $\chi(B_i,B_j)=0$ for $i>j$, it is easy to check that
\ben
\widetilde{B}_i= [L_{E_1}L_{E_2}\cdots L_{E_{i-1}} (E_i)] =
\sigma_1^{-1} \sigma_2^{-1} \cdots \sigma_{i-1}^{-1}(B_i).
\een
Note that the relations in the above formula are identical to the relations in \eqref{half-twist}. Therefore, if we define $B_i^*$ and $\widetilde{B}_i^*$ such that
$h(B_i^*,B_j)=\delta_{ij}$ and $h(\widetilde{B}_i^*, \widetilde{B}_j^*)=\delta_{ij}$, then we have a relation corresponding to \eqref{dual-half-twist}
\ben
\widetilde{B}_k^*=B_k^* + \sum_{a=k+1}^N q^{-1} h(B_k,B_a) B_a^*.
\een
Note that $q^{-1} h(B_k,B_a)=\chi(B_k,B_a)$ because $\chi(B_a,B_k)=0$ for $a>k$.
Let $\chi_{B}$ be the Gram matrix of $\chi$ in the basis $B_1,\dots,B_N$, that is, $\chi_{B,ij}=\chi(B_i,B_j)$. Similarly, let $\chi_{\widetilde{B},ij}=\chi(\widetilde{B}_i,\widetilde{B}_j)$ be the Gram matrix of $\chi$ in the basis $(\widetilde{B}_1,\dots,\widetilde{B}_N)$. Suppose that $T=(T_{ij})$ is the matrix describing the transition between the two bases: $\widetilde{B}_j=\sum_{i=1}^N B_iT_{ij}$. Then we have
\ben
\delta_{kj} =h(\widetilde{B}_k^*, \widetilde{B}_j) = \sum_{a,i=1}^N
\chi_{B,ka} T_{ij}h(B_a^*, B_i) = \sum_{i=1}^N \chi(B_k,B_i) T_{ij}
\een
which implies that 
\ben
\chi(B_i,\widetilde{B}_j) = \sum_{k=1}^N \chi(B_i,B_k) T_{kj}=\delta_{ij}.\qed
\een
\begin{definition}
A unimodular Mukai lattice is a finitely generated free $\mathbb{Z}$-module $V$ with a unimodular bilinear form (not necessarily symmetric) $\langle .,.\rangle: V \times V \rightarrow \mathbb{Z}$. \\
An element $e \in V$ is called exceptional if $\langle e,e\rangle = 1$. A $\mathbb{Z}$-basis $(e_1, \dots, e_n)$ of the Mukai lattice is called exceptional if  $\langle e_i,e_i \rangle = 1$, $\forall i$ and  $\langle e_j,e_ i \rangle = 0$ for $j > i$.\qed
\end{definition}
\begin{remark}
The projection on $K_0(\mathcal{T})$ of a full exceptional collection in $\mathcal{T}$ is an exceptional basis. The pair ($K_0(\mathcal{T})$, $\chi$), where $\chi$ is the Grothendieck-Euler-Poincaré pairing defined above, is a unimodular Mukai lattice. The matrix $G$ whose $(i,j)$-entry is $\chi(E_i,Ej)$ is called the Gram matrix associated to the exceptional collection $(E_1, \dots, E_n).$\qed
\end{remark}
We are interested in the case where the triangulated category $\mathcal{T}$ is the bounded derived category of coherent sheaves of a smooth algebraic variety $X$, which we denote by $D^b(X)$. It is interesting to know in which cases $D^b(X)$ has a full exceptional collection. 

\subsection{Original formulation of Dubrovin conjecture} 

For a smooth Fano variety we have formulated two questions. The first one is about the semisimplicity of the big quantum cohomology and the second one is about the existence of a full exceptional collections in $D^b(X)$. At first sight, the questions seem unrelated but this is not the case. In fact, the main content of Dubrovin conjecture is precisely that the answer of these two questions should be linked. In his ICM talk in 1998, following a proposal of Alexey Bondal, Dubrovin proposed the following conjecture (see \cite{Dubrovin1998} and also \cite{CDG2024}).
\begin{conjecture}[Dubrovin 1998]\label{conj:Du1998}
Let $X$ be a Fano variety. 
\begin{enumerate}
\item[(1)] 
The big quantum cohomology $QH^*(X)$ is semisimple if and only if $D^b(X)$ admits a full exceptional collection $(E_1, \dots, E_N)$, where $N = \mathrm{dim}_\CC H^*(X)$. 
\item[(2)] 
The Stokes matrix for the first structure connection $S = (s_{ij})$ is equal to the Gram matrix for $(E_1, \dots, E_N)$, i.e., $s_{ij} = \chi(E_i,E_j)$.
\item[(3)] 
The central connection matrix $C$ has the form $C = C'C''$, where the columns of $C''$ are the components of $\mathrm{ch}(E_j) \in H^+(X)$ and $C':H^*(X) \rightarrow H^*(X)$ is some operator satisfying $C'(c_1(X)a) = c_1(X)C'(a)$ for any $a \in H^*(X)$.
\end{enumerate}
\end{conjecture} 

There are two important developments that led to a modification of the conjecture. First of all, it was suggested by Arend Bayer (see \cite{Ba2004}) that the Fano condition is not important, so it should be dropped. Second, a precise statement about the central connection matrix, was proposed independently in \cite{GGI2016} by Galkin-Golyshev-Iritani and in \cite{CDG2024} by Cotti-Dubrovin-Guzzetti. An important role in this refinement is played by the so-called Gamma class. 

\subsection{Refined version of the conjecture} \label{sec:rdc}

Let $X$ be a smooth projective variety of complex dimension $D$. The cohomology class (see \cite{GGI2016}) $\widehat{\Gamma}_X =\widehat{\Gamma}^+_X:= \prod_{i=1}^D \Gamma(1 + \delta_i)$, where $\delta_1, \dots, \delta_D$ are the Chern roots of $TX$ and $\Gamma(X)$ is the Gamma function, is called the Gamma class. Following \cite{CDG2024}, we introduce also the class $\widehat{\Gamma}_X^{-} = \prod_{i=1}^D \Gamma(1 - \delta_i)$.
%This class plays a role in the comparison between the proposals in \cite{GGI2016} and \cite{CDG2024}. \\
\begin{remark}
  Gamma classes appear in the study of integral structures of quantum cohomology in the work of Iritani \cite{Iritani} and the work of Katzarkov-Kontsevich-Pantev \cite{KKP} on noncommutative Hodge structures. It also appears in physics under the hemisphere partition functions studied by Hori-Romo in \cite{HoriRomo}.\qed
\end{remark}
\begin{remark}
The Gamma class $\widehat{\Gamma}_X$ can be expanded as 
\ben
\widehat{\Gamma}_X = \mathrm{exp}\Big(-C_{eu} c_1(X) + \sum_{k \geqslant 2} (-1)^k (k-1)! \zeta(k) \operatorname{ch}_k(TX)\Big)
\een
where $C_{eu}$ is the Euler constant. This is obtained from the Taylor expansion for the Gamma function. \qed
\end{remark}
The other ingredient introduced in \cite{CDG2024} is given by two morphisms $\mathfrak{A}^{\pm}_X : K_0(X) \rightarrow H^*(X, \mathbb{C})$ which are defined as follows. 
Let $E \in D^b(X)$. Since $X$ is smooth, the object $E$ is isomorphic to a bounded complex of locally free sheaves $F^{\bullet}$, therefore a graded version of the Chern character can be defined as $\operatorname{Ch}(E):= \sum_j (-1)^j \operatorname{Ch}(F^j)$ where $\operatorname{Ch}(F^j) = \sum_\alpha e^{2\pi \ii \alpha}$ where the sum  is over the Chern roots $\alpha$ of $F^j$. Note that the standard Chern character is 
$\operatorname{ch}(E):= \sum_j (-1)^j \operatorname{ch}(F^j)$ where 
$\operatorname{ch}(F^j) = \sum_\alpha e^{\alpha}$. In other words, the difference between $\op{Ch}$ and $\op{ch}$ is in re-scaling each Chern root by $2\pi\ii$. The morphisms $\mathfrak{A}_X^{\pm}$ are defined as follows: 
\beq\label{A_X}
\mathfrak{A}_X^{\pm}(E) = \frac{\ii^{\overline{D}}}{(2\pi)^{\frac{D}{2}}} \widehat{\Gamma}^{\pm}(X) \cup 
\mathrm{exp}(\pm \pi \ii c_1(X)) \cup 
\operatorname{Ch}(E),
\eeq
where $\overline{D} \in \{0,1\}$ is the remainder of the division of $D$ by $2.$ In order to define the monodromy data of a Frobenius manifold, Cotti--Dubrovin--Guzzetti have introduced chambers $\Omega_\ell$ for every oriented line $\ell\subset \CC$ with orientation specified by a unit vector $e^{\ii \phi}$, $\phi\in [0,2\pi)$. In our notation, $\Omega_\ell$ is an open subset of the Frobenius manifold $M\subset H^*(X,\CC)$ consisting of semisimple points $t$ such that 
\begin{enumerate}
\item[(i)]
The canonical coordinates $u_1(t),\dots,u_N(t)$ are pairwise distinct. 
\item[(ii)]
The vector $\eta:=\ii e^{\ii \phi}$ is an admissible direction, that is, $\ii e^{\ii\phi}$ is not parallel to $u_i-u_j$ for all $i\neq j$. 
\end{enumerate}
In every chamber $\Omega_\ell$, the canonical coordinates are enumerated according to the so-called {\em lexicographical order}: if $i<j$ then $\operatorname{Re}(u_i-u_j) e^{-\ii\phi}<0$, or equivalently if we stand at $u_i$ and look in the admissible direction $\eta:=\ii e^{\ii \phi}$, then $u_j$ is on our right. The Frobenius manifold underlying quantum cohomology has a natural calibration given by the $S$-matrix $S(t,z)=1+S_1(t) z^{-1} + S_2(t) z^{-2}+\cdots$ where $S_k(t)\in \operatorname{End}(H^*(X,\CC))$ are defined by 
\ben
(S_k(t)\phi_a,\phi_b):= 
\langle
\phi_a\psi^{k-1},\phi_b \rangle_{0,2}(t).
\een
The nilpotent operator $\rho:=c_1(TX)\cup$. The monodromy data is defined as explained in Sections \ref{sec:AE} and \ref{sec:ccm}. Namely, there are unique solutions $X(\pm\eta,t,z)\sim \Psi(t) R(t,z) e^{U/z}$ as $z\to 0$ holomorphic for $z\in H_{\pm\eta}$ where $H_{\eta}$ (resp. $H_{-\eta}$) is the right (resp. left) half-plane bounded by the oriented line $\ell$. The 3 solutions to the quantum connection $X(-\eta,t,z)$, $X(\eta,t,z)$, and $S(t,z)z^\theta z^{-\rho}$ are analytic in $z$ in a sector containing the positive part of the line $\ell$ and hence we can define the Stokes matrix $V_+$ and the central connection matrix $C$ by formulas \eqref{con-m}--\eqref{stokes-m}
%\ben
%X(-\eta,t,z)= X(\eta,t,z) V_+,\quad 
%X(-\eta,t,z)= S(t,z) z^\theta z^{-\rho} C^{-1}. 
%\een
\begin{conjecture}
  [Refined Dubrovin conjecture 2024, see conjecture 5.2 in \cite{CDG2024}]
  \label{conj:rdc}
Let $X$ be a smooth Fano variety of Hodge-Tate type, then
\begin{enumerate}
\item[(1)] 
The big quantum cohomology $QH^*(X)$ is semisimple if and only if there exists a full exceptional collection in $D^b(X)$.
\item[(2)] 
If $QH^*(X)$ is semisimple and convergent, then for any oriented line $\ell$ (of slope $\phi \in [0, 2 \pi)$) in the complex plane, there is a correspondence between $\ell$-chambers and helices with a marked foundation $(E_1, \dots, E_N)$ in $D^b(X)$. 
\item[(3)] 
The monodromy data computed in an $\ell$-chamber $\Omega_{\ell}$, in the lexicographical order, is related to the following geometric data of the corresponding exceptional collection $(E_1, \dots, E_N)$ (the marked foundation): 
\begin{enumerate}
\item[(3a)] 
The Stokes matrix $V_+$ is equal to the Gram matrix of the Grothendieck-Poincaré-Euler product on $K_0(X)_{\mathbb{C}}$, computed with respect to the exceptional basis $([E_1], \dots, [E_N])$, that is, $V_{+,ij} = \chi (E_i, E_j)$.
\item[(3b)] 
The inverse central connection matrix $C^{-1}$ coincides with the matrix associated with the $\mathbb{C}$-linear morphism $\mathfrak{A}^{-}_X : K_0(X)_{\mathbb{C}} \rightarrow H^*(X, \mathbb{C})$ defined above -- see \eqref{A_X}.
The matrix is computed with respect to the exceptional basis $([E_1], \dots, [E_N])$ and any pre-fixed cohomological basis $\{\phi_{\alpha}\}_{\alpha=1}^N$.
\end{enumerate}
\end{enumerate}
\end{conjecture}
Let us comment on the correspondence between our notation and the notation in \cite{CDG2024}. Similarly to Remark \ref{rem:Du1999}, our spectral parameter $z$ is the inverse of the one in \cite{CDG2024}. The involution $z\mapsto z^{-1}$ switches the roles of left and right. Therefore, our Stokes matrix $V_+$ corresponds to $S^{-1}$ in \cite{CDG2024} (see Definition 2.24 and Theorem 2.25 in \cite{CDG2024}). The inverse of our connection matrix $C^{-1}$ corresponds to one of the connection matrices $C^{(k)}$ in \cite{CDG2024} (see Definition 2.24 in \cite{CDG2024}). Here $k\in \ZZ$ is an integer parametrizing the possible choices of a branch of $\log z$. The precise value of $k$ is not important for the following reason. Shifting the superscript $k\mapsto k+1$ is equivalent to the action of the \emph{Serre functor} on the exceptional collection, that is, 
$(E_1,\dots,E_N)\mapsto (S_X(E_1),\dots,S_X(E_N))$ where 
\ben
S_X: \mathcal{D}^b(X)\to \mathcal{D}^b(X),\quad 
S_X(E) = E[D]\otimes K_X
\een
where $K_X$ is the canonical bundle of $X$. In other words, if the refined Dubrovin conjecture holds for some value of $k$, then it holds for all values of $k$.
\begin{remark}
Some parts of the Dubrovin conjecture, in its original or refined forms, have been verified for several Fano varieties, see \cite{CDG2024} for a detailed account about the cases where the conjecture has been proved. \qed
\end{remark}

\subsection{Exceptional collections, reflection vectors, and Dubrovin conjecture}\label{sec:rv-dc}

Motivated by the definition of a distinguished basis in singularity theory (see \cite{AGuV,Eb2007}), let us define a distinguished system of reference paths. Recall that we have fixed a reference point $(t^\circ,\lambda^\circ)$, such that 
$|u_i^\circ|<|\lambda^\circ|$ for all $i$ and $\op{Re}(u_i^\circ)\neq \op{Re}(u_j^\circ)$ for all $i\neq j$ where $u_i^\circ:=u_i(t^\circ)$ are the canonical coordinates of $t^\circ$. Let $\Delta$ be the disk with center $0$ and radius $\lambda^\circ$ (recall that $\lambda^\circ$ is a positive real number). 
\begin{definition}\label{def:db}
A system of paths $(C_1,\dots,C_N)$ inside $\Delta$ is said to be a {\em distinguished system} of reference paths if 
\begin{enumerate}
\item[(i)]
The path $C_i$ has no self-intersections and it connects $\lambda^\circ$ with one of the points $u_1^\circ,\dots,u_N^\circ$.
\item[(ii)]
For each pair of paths $C_i$ and $C_j$ with $i\neq j$, the only common point is $\lambda^\circ$.
\item[(iii)]
The paths $C_1,\dots,C_N$ exit the point $\lambda^\circ$ in an anti-clockwise order counted from the boundary of the disk $\Delta$.\qed
\end{enumerate}
\end{definition}
Two distinguished systems of reference paths $C'=(C_1',\dots,C_N')$ and $C''=(C''_1,\dots,C''_N)$ will be considered homotopy equivalent if there exists a continuous family $C(s)=(C_1(s),\dots,C_N(s))$, $s\in [0,1]$ such that $C(s)$ is a distinguished system of reference paths $\forall s\in [0,1]$ and $C(0)=C'$ and $C(1)=C''$ (for more details see \cite{Eb2007}, Section 5.7). The braid group on $N$ strings acts naturally on the set of homotopy equivalence classes of distinguished reference paths. Namely, we have the operations  
\ben
L_i (C_1,\dots, C_N):= (C_1,\dots, C_{i-1}, L_{C_i}C_{i+1}, C_i,\dots, C_N),\quad
1\leq i\leq N-1,
\een  
where $L_{C_i}C_{i+1}$ is a small perturbation of the composition of $C_{i+1}$ and the anti-clockwise simple loop corresponding to $C_i$. Similarly, we have the operation 
\ben
R_i(C_1,\dots,C_N):= 
(C_1,\dots, C_{i-1}, C_{i+1}, R_{C_{i+1}}C_i, C_{i+2},\dots,C_N),\quad 
1\leq i\leq N-1,
\een
where $R_{C_{i+1}}C_{i}$ is a small perturbation of the composition of $C_{i}$ and the clockwise simple loop corresponding to $C_{i+1}$. The operation $R_i$ is inverse to $L_i$ and the following braid group relations are satisfied:
\ben
L_iL_{i+1}L_i= L_{i+1}L_iL_{i+1},\quad 
R_iR_{i+1}R_i= R_{i+1}R_iR_{i+1},\quad 1\leq i\leq N-2.
\een
The braid group on $N$ strings acts transitively on the set of  homotopy equivalence classes of distinguished reference paths. This is almost an immediate consequence of the definition of a braid (see \cite{Eb2007}, Proposition 5.15).

The homotopy class of a distinguished system of reference paths $C=(C_1,\dots,C_N)$ determines a set of non-twisted reflection vectors $\{\beta_1,\dots,\beta_N\}$. In order to determine a set of twisted reflection vectors we have to make a choice for the value of $\log (\lambda -u_i)$ for all $\lambda\in C_i$ sufficiently close $u_i$. To this end,  we fix an admissible direction $\eta$ and a value of $\log \eta$ (Cf. Section \ref{sec:mdrv}). We require that each path $C_i$ approaches $u_i$ along a straight segment $[u_i,u_i+\epsilon \eta]$ where $\epsilon >0$ is sufficiently small. The branch of $\log (\lambda-u_i)$ is chosen so that the value of $\log (\lambda-u_i)$ at $\lambda=u_i+\epsilon \eta$ is $\ln \epsilon + \log \eta$.
Note that the mutation operations $L_i$ and $R_i$ preserve the above form of $C$, that is, each path $C_i$ approaches $u_i$ in the direction $\eta$. The braid group still acts transitively on the homotopy classes of distinguished systems of such paths. 
\begin{remark}
In singularity theory, one usually requires the order of the distinguished system of reference paths to be clockwise. We change it to anti-clockwise in order to have an agreement between the lexicographical order defined by an admissible direction. \qed 
\end{remark}
\begin{remark}
  The operations $L_i$ and $R_i$ are usually denoted by $\alpha_i$ and $\beta_{i+1}$. Our notation is motivated by the corresponding notation for left and right mutations in the case of exceptional collections. \qed
\end{remark}
Partially motivated by the work of Milanov--Xia (see \cite{MilanovXia}, Conjecture 1.6), we would like to propose the following conjecture.
\begin{conjecture} \label{conj1} 
If $\beta_1,\dots,\beta_N$ is a set of non-twisted reflection vectors corresponding to a system of distinguished reference paths, then there exists a full exceptional collection $(F_1, \dots, F_N)$ in $D^b(X)$ such that $\beta_i = \Psi_Q(F_i)$ for all $i$. 
\end{conjecture}
Conjecture \ref{conj1} is weaker than Conjecture 1.6 in \cite{MilanovXia}. Namely, the proposal in \cite{MilanovXia} is that every full exceptional collection determines a set of reflection vectors. This is more difficult to prove. Nevertheless, it is expected that any two full exceptional collections are related by a sequence of mutations. If this expectation is correct then Conjecture 1.6 in \cite{MilanovXia} is equivalent to Conjecture \ref{conj1}. 
\begin{theorem}\label{thm:refl_Du} Conjecture \ref{conj1} is equivalent to the refined Dubrovin conjecture.
\end{theorem}
\begin{proof} 
Let us consider the case when all Novikov variables $q_1=\cdots=q_r=1$. The general case can be obtained by applying the divisor equation. 
Suppose that the distinguished system of reference paths is given by $C_i(\eta)$ ($1\leq i\leq N$) where $\eta$ is an admissible direction. 
Assuming that the refined Dubrovin conjecture holds, let us derive the formulas for the reflection vectors in Conjecture \ref{conj1}. Note that the exceptional collection $F_1,\dots, F_N$ will be different from $E_1,\dots,E_N$. The inverse central connection  matrix defines a map $C^{-1}: \CC^N\to H^*(X,\CC)$ which according to the refined Dubrovin conjecture is given by 
\ben
C^{-1}(e_i)=
\frac{\ii^{\overline{D}}}{(2\pi)^{D/2}}
\widehat{\Gamma}^-_X \cup e^{-\pi\ii\rho}\cup 
\operatorname{Ch}(E_i). 
\een
If $\phi_a\in H^*(X,\CC)$ is one of the basis vectors, then we have
\ben
\phi_a=C^{-1}(C(\phi_a))=
\frac{1}{\sqrt{2\pi}}
\sum_{i=1}^N C^{-1}(e_i) (\beta_i,\phi_a),
\een 
where we used the formula for the $(i,a)$-entry of $C$ from Theorem \ref{thm:refl_cc}, part a). We get 
\ben
\phi_a= 
\frac{\ii^{\overline{D}}}{(2\pi)^{(1+D)/2}}
\sum_{i=1}^N 
\widehat{\Gamma}^-_X \cup e^{-\pi\ii\rho}\cup 
\operatorname{Ch}(E_i)
(\beta_i,\phi_a).
\een
Let us multiply the above identity, using the classical cup product, by 
$\widehat{\Gamma}^+_X \cup \operatorname{Ch}(F_j)$ where $F_j\in K^0(X)$ will be specified later on. Recall that 
\ben
\widehat{\Gamma}^+_X
\widehat{\Gamma}^-_X =\prod_\delta \Gamma(1+\delta)\Gamma(1-\delta) =
\prod_\delta \frac{2\pi\ii \delta}{e^{\pi\ii\delta}-e^{-\pi\ii\delta}} = 
e^{-\pi \ii \rho} (2\pi\ii)^{\rm deg} \operatorname{Td}(X),
\een
where the product is over all Chern roots of $TX$ and $\operatorname{Td}(X)$ is the Todd class of $TX$. Note also that $e^{-2\pi \ii c_1(TX)} = \operatorname{Ch}(K)$ where $K$ is the canonical bundle of $X$. We get 
\ben
\widehat{\Gamma}^+_X \cup 
\operatorname{Ch}(F_j)\cup \phi_a=
\frac{\ii^{\overline{D}}}{(2\pi)^{(1+D)/2}}
\operatorname{Ch}(F_j\otimes K)\, 
(2\pi\ii)^{\rm deg}( \operatorname{Td}(X))
\sum_{i=1}^N 
\operatorname{Ch}(E_i)
(\beta_i,\phi_a).
\een
Let us integrate the above identity over $X$. For dimensional reasons, we can replace the expression 
$\operatorname{Ch}(F_j\otimes K)\, 
(2\pi\ii)^{\rm deg}( \operatorname{Td}(X))
\operatorname{Ch}(E_i)$ with 
$(2\pi\ii)^D
\operatorname{ch}(F_j\otimes K)\, 
\operatorname{Td}(X)
\operatorname{ch}(E_i).$ 
Recalling the Hierzerbruch--Riemann--Roch formula we get 
\ben
(\widehat{\Gamma}^+_X \cup \operatorname{Ch}(F_j), \phi_a) =
\ii^{\overline{D} + D} (2\pi)^{(D-1)/2}
\sum_{i=1}^N \chi(E_i\otimes F_j\otimes K) (\beta_i,\phi_a).
\een
Since the Poincare pairing is non-degenerate, the above formula implies 
\beq\label{beta-eq}
\widehat{\Gamma}^+_X \cup \operatorname{Ch}(F_j) = 
\ii^{\overline{D} + D} (2\pi)^{(D-1)/2}
\sum_{i=1}^N \chi(E_i\otimes F_j\otimes K) \, \beta_i.
\eeq
Let us recall that by Serre duality, that is, $H^i(X,E^\vee)\cong H^{D-i}(X,E \otimes K)^\vee$, we have 
$\chi(E\otimes K)=(-1)^D\chi(E^\vee).$ Therefore, 
\ben
\ii^{\overline{D} + D}\chi(E_i\otimes F_j\otimes K)= 
\ii^{\overline{D} - D}\chi(E_i^\vee\otimes F^\vee_j)= 
(-1)^{(\overline{D} - D)/2} 
\langle F_j, E_i^\vee \rangle.
\een
Clearly, we can choose $K$-theoretic vector bundles $F_1,\dots,F_N$ such that $\langle F_j, E_i^\vee \rangle = (-1)^{(\overline{D} - D)/2}\, \delta_{i,j}$. Once this choice is made, we get from \eqref{beta-eq} that $\beta_j=\Psi(F_j)$ where $\Psi=\left.\Psi_q\right|_{q_1=\cdots=q_r=1}$. We claim that $F_1,\dots,F_N$ are the K-theoretic classes of a full exceptional collection. Indeed, recalling Lemma \ref{le:left_KD}, it is sufficient to choose a full exceptional collection $(G_1,\dots,G_N)$ such that its left Koszul dual is $(E_N^\vee,\dots,E_1^\vee)$. Once we do this, we can simply put  $F_i=G_i[-l]$ where $l:=\tfrac{D-\overline{D}}{2}=\left[\tfrac{D}{2}\right]$, that is, $l$ is the integral part of $\tfrac{D}{2}$. In order to define $G_i$, we simply have to invert the sequence of mutation operations that define the left Koszul dual, that is, 
\ben
(G_1,\dots,G_N)=
(R_{N-1}R_{N-2}\cdots R_1)
(R_{N-1}R_{N-2}\cdots R_2)\cdots (R_{N-1}R_{N-2}) R_{N-1} 
(E_N^\vee,\dots,E_1^\vee).
\een
In order to complete the proof that the refined Dubrovin conjecture implies Conjecture \ref{conj1} we need only to recall the braid group action. More precisely, note that after a small perturbation the system of reference paths $(C_1(\eta),\dots,C_N(\eta))$ corresponding to an admissible direction $\eta$ becomes a distinguished system of reference paths. According to Theorem \ref{thm:refl_cc}, the (non-twisted) reflection vectors $\beta_1,\dots,\beta_N$ coincide with the twisted reflection vectors $\beta_1(m),\dots,\beta_N(m)$ corresponding to the reference paths $C_1(\eta),\dots,C_N(\eta)$. We will consider the braid group action on $\beta_1(m),\dots,\beta_N(m)$, that is, we will view $\beta_1,\dots,\beta_N$ as twisted reflection vectors. It turns out, that the resulting collections of twisted reflection vectors are independent of $m$, that is, we will obtain all collections of non-twisted reflection vectors corresponding to distinguished systems of reference paths. Let $\gamma=(C_1,\dots,C_N)$ be a distinguished system of reference paths obtained from $(C_1(\eta),\dots,C_N(\eta))$ via the braid group action. Suppose that Conjecture \ref{conj1} holds for $\gamma$ and that the following extra conditions hold: 
\begin{enumerate}
\item[(i)] The twisted reflection vectors $\beta_1(m),\dots,\beta_N(m)$ corresponding to $\gamma$ are independent of $m$; 
\item[(ii)] We have 
$\langle \beta_i,\beta_j\rangle=0$ for $i>j$.
\end{enumerate}
Then we claim that the conjecture holds for $L_i\gamma$ and moreover 
the twisted reflection vectors $\widetilde{\beta}_1,\dots, \widetilde{\beta}_N$ corresponding to $L_i\gamma$ satisfy the extra conditions (i) and (ii). Note that 
\ben
\widetilde{\beta}_i=\sigma_i^{-1}(\beta_{i+1}),\quad
\widetilde{\beta}_{i+1} = \beta_i,\quad
\widetilde{\beta}_k=\beta_k,\quad k\neq i,i+1.
\een
We have
\ben
\widetilde{\beta}_i=
\sigma_i^{-1}(\beta_{i+1}) = \beta_{i+1}-qh_m(\beta_{i+1},\beta_i) \beta_i=
\beta_{i+1}-\langle \beta_i,\beta_{i+1}\rangle \beta_i. 
\een
Using the above formulas for $\widetilde{\beta}_1,\dots,\widetilde{\beta}_N$ we get that both conditions (i) and (ii) are satisfied. Furthermore, recalling formula \eqref{Psi_eu}, we get that if $\beta_k=\Psi(F_k)$ for some exceptional collection $\phi=(F_1,\dots,F_N)$, then $\widetilde{\beta}_k=\Psi(\widetilde{F}_k)$ where $(\widetilde{F}_1,\dots, \widetilde{F}_N)=L_i\phi$, that is, Conjecture \ref{conj1} holds for $L_i\gamma$.

Finally, for the inverse statement, that is, Conjecture \ref{conj1} implies the refined Dubrovin conjecture, one just has to go backwards. We leave the details as an exercise.  
\end{proof}

Let us point out the following important property of a distinguished bases which was obtained as a byproduct of the proof of Theorem \ref{thm:refl_Du}.
\begin{proposition}\label{prop:dist_eu}
Let $\beta_1,\dots,\beta_N$ be a set of reflection vectors corresponding to a distinguished system of reference paths. Then the Gram matrix of the Euler pairing is upper triangular: $\langle \beta_i,\beta_j\rangle =0$ for all $i>j$ and
$\langle \beta_i,\beta_i\rangle =1$.  
\end{proposition}
Indeed, if the reference paths correspond to an admissible direction, then the statement was already proved in Theorem \ref{thm:refl_cc}, part c). In the proof of Theorem \ref{thm:refl_Du} we proved that the statement is invariant under the action of the braid group. Therefore, since the braid group acts transitively on the set of distinguished system of reference paths, the statement of Proposition \ref{prop:dist_eu} is clear. 

\begin{remark}
  Galkin--Golyshev--Iritani proposed in \cite{GGI2016} the so-called $\Gamma$-Conjecture II. Roughly this conjecture says that the columns of the central connection matrix are the components of $(2\pi)^{-\tfrac{D}{2}}\widehat{\Gamma}^+_X \cup \op{Ch}(E_i)$, where $D= \mathrm{dim}_{\mathbb{C}} X$, for an exceptional collection $(E_1, \dots, E_N)$. It was proved in \cite{CDG2024} that $\Gamma$-conjecture II is equivalent to part \emph{(3b)} of the refined Dubrovin conjecture. Moreover, it was proved  in \cite{GGI2016} that $\Gamma$-conjecture II implies part \emph{(3a)} of the refined Dubrovin conjecture. Therefore, Conjecture \ref{conj1} is also equivalent to $\Gamma$-conjecture II.
  \qed
\end{remark}
\begin{remark}
  Halpern-Leistner proposed in \cite{Halpern-Leistner}  the so-called noncommutative minimal model program. This is a set of conjectures about canonical paths on the space of stability conditions $\mathrm{Stab}(X)/\mathbb{G}_a$ that imply previous conjectures about $D^b(X)$. In particular, it implies one direction of Dubrovin conjecture regarding the existence of exceptional collections. It would be interesting to see the relations between the canonical paths in the noncommutative minimal model program and the reflection vectors corresponding to a system of distinguished reference paths in Conjecture \ref{conj1}.
  \qed
\end{remark}

\section{Acknowledgements}
We are thankful to Alexey Bondal for many useful discussions on the Dubrovin conjecture and especially for pointing out to us the notion of left and right Koszul dual of an exceptional collection. The first author also thanks Jin Chen and Mauricio Romo for many interesting discussions on the $\Gamma$-conjectures and Dubrovin conjecture. We are also thankful to the anonymous referee for very useful suggestions that helped us improve the exposition. This work is supported by the World Premier International Research Center Initiative (WPI Initiative), MEXT, Japan and by JSPS Kakenhi Grant Number JP22K03265. The first author is also supported by the President's International Fellowship Initiative of the Chinese Academy of Sciences.

\bibliographystyle{plain}
\bibliography{main}

% --- Affiliation Block ---
\bigskip
\noindent
\textbf{Todor Milanov} \\
Kavli IPMU (WPI), UTIAS, The University of Tokyo\\
Kashiwa, Chiba, 277-8583, Japan \\
Email: \texttt{todor.milanov@ipmu.jp}

\bigskip
\noindent
\textbf{John Alexander Cruz Morales} \\
%\ead{jacruzmo@unal.edu.co}
%% \ead[url]{home page}
%\cortext[cor1]{Corresponding author}
Departamento de Matemáticas, Universidad Nacional de Colombia\\
Ciudad Universitaria, Bogotá, Colombia\\
School of Mathematical Sciences, University of Science and Technology of China \\
Hefei, 230026, P.R. China\\
Email: \texttt{jacruzmo@unal.edu.co}

\end{document}